\documentclass[a4paper,10pt]{amsart}

\usepackage{amssymb,esint,graphicx,subfigure}
\usepackage{amsmath}

\newtheorem{theorem}{Theorem}[section]
\newtheorem{lemma}[theorem]{Lemma}

\newtheorem{corollary}[theorem]{Corollary}
\newtheorem{definition}{Definition}[section]

\theoremstyle{remark}
\newtheorem{remark}[theorem]{Remark}
\theoremstyle{definition}
\newtheorem{example}{Example}[section]

\numberwithin{equation}{section}

\newcommand{\ra}{\rightarrow}

\newcommand{\alp}{\alpha}

\newcommand{\R}{\ensuremath{\mathbb{R}}}
\newcommand{\N}{\ensuremath{\mathbb{N}}}

\newcommand{\Div}{\mathrm{div}}
\newcommand{\Levy}{\ensuremath{\mathcal{L}}}

\newcommand{\sgn}{{\rm sgn}\, }

\newcommand{\dif}{\mathrm{d}}

\begin{document}

\title[Continuous dependence estimates for integro-PDEs]{Continuous
dependence estimates for nonlinear fractional convection-diffusion
equations}

\author[N.~Alibaud]{Natha\"{e}l Alibaud}
\address[Natha\"{e}l Alibaud]{UMR CNRS 6623, Universit\'e de Franche-Comt\'e\\
16 route de Gray\\ 25 030 Besan\c{c}on cedex, France\\ 
and
\'Ecole Nationale Sup\'erieure de M\'ecanique et des Microtechniques, 26 chemin de l'\'Epitaphe, 25030 Besan\c{c}on cedex, France\\
and Department of Mathematics, Faculty of Science\\
Prince of Songkla University\\
Hat Yai, Songkhla, Thailand, 90112}
\email{nathael.alibaud\@@{}ens2m.fr}

\author[S.~Cifani]{Simone Cifani}
\address[Simone Cifani]{Department of Mathematics\\
Norwegian University of Science and Technology (NTNU)\\
N-7491 Trondheim, Norway} \email[]{simone.cifani\@@{}math.ntnu.no}
\urladdr{http://www.math.ntnu.no/\~{}cifani/}

\author[E.~R.~Jakobsen]{Espen R. Jakobsen}
\address[Espen R. Jakobsen]{Department of Mathematics\\
Norwegian University of Science and Technology (NTNU)\\
N-7491 Trondheim, Norway} \email[]{erj\@@{}math.ntnu.no}
\urladdr{http://www.math.ntnu.no/\~{}erj/}



\thanks{This research was supported by the Research Council of Norway (NFR), through the project ``Integro-PDEs:
Numerical methods, Analysis, and Applications to Finance'', and by the ``French ANR project CoToCoLa''.}

\begin{abstract}
  We develop a general
framework for finding error
  estimates for convection-diffusion
  equations with nonlocal, nonlinear, and possibly degenerate
  diffusion terms.  The equations are nonlocal because they involve
  fractional diffusion operators that are generators of pure jump
  L\'{e}vy processes (e.g.~the fractional Laplacian). As an application,
  we derive continuous   dependence estimates on the nonlinearities
  and on the L{\'e}vy measure of the diffusion term. Estimates of
  the rates of convergence   for general nonlinear nonlocal vanishing
  viscosity approximations of
  scalar conservation laws then follow as a corollary. Our
  results both cover, and extend to new equations, a
  large part of the known error estimates in the literature.
\end{abstract}

\maketitle

\section{Introduction}
This paper is concerned with the following Cauchy problem:
\begin{equation}\label{1}
\begin{cases}
\partial_t u(x,t)+\Div \left(f(u) \right)(x,t)
=\Levy^\mu[A(u(\cdot,t))](x) &\text{in} \quad  Q_T:=\mathbb{R}^d\times(0,T),\\
u(x,0)=u_{0}(x), &\text{in} \quad \R^d,
\end{cases}
\end{equation}
where~$u$ is the scalar unknown function,~$\Div$ denotes the
divergence with respect to (w.r.t.)~$x$, and the
operator~$\Levy^\mu$ is defined for all~$\phi \in C^\infty_c(\R^d)$
by
\begin{align}\label{nat:levy-form}
\Levy^\mu[\phi](x):=\int_{\R^d \setminus \{0\}}
\phi(x+z)-\phi(x)-z\cdot D \phi(x)\mathbf{1}_{|z| \leq 1}\
\dif\mu(z),
\end{align}
where $D \phi$ denotes the gradient of~$\phi$ w.r.t.~$x$ and $
\mathbf{1}_{|z| \leq 1}=1 $ for $|z|\leq1$ and $=0$ otherwise.
Throughout the paper, the data $(f,A,u_0,\mu) $ is assumed to
satisfy the following assumptions:
\begin{align}
&f \in W^{1,\infty}(\R,\R^d) \text{ with } f(0)=0,\label{flux}\\
&A\in W^{1,\infty}(\R) \text{ is nondecreasing with } A(0)=0,\label{Aflux}\\
&u_0\in L^\infty(\R^d)\cap L^1(\R^d)\cap
BV(\R^d),\label{in_condition}\\
\intertext{and} &\text{$\mu$ is a nonnegative Radon measure on~$\R^d
\setminus \{0\}$
  satisfying}\qquad\qquad \qquad
\label{levy-measure}
\end{align}
$$\int_{\R^d \setminus \{0\}} |z|^2 \wedge 1 \, \dif\mu(z)<+\infty,$$
where we use the notation~$a\wedge b=\min \{a,b\}$. The measure
$\mu$ is a L\'evy measure.
\begin{remark}\*
\label{rem1}
\begin{enumerate}
\item Subtracting constants to~$f$ and~$A$ if necessary, there is no loss of generality in assuming that~$f(0)=0$ and~$A(0)=0$.
\item Our results also hold for locally Lipschitz-continuous
  nonlinearities~$f$ and~$A$ since solutions will be bounded; see
  Remark~\ref{nat:rem-loc-lip} for more details.
\item Assumption \eqref{levy-measure} and a Taylor expansion reveal
  that $\Levy^\mu[\phi]$ is well-defined for e.g.~bounded $C^2$
  functions $\phi$:
$$
|\Levy^\mu[\phi](x)| \leq \max_{|z| \leq 1} |D^2\phi(x+z)|
\int_{0<|z|\leq 1}\frac12 |z|^2  \dif\mu(z)+2\|\phi\|_{L^\infty}
\int_{|z|>1}\dif\mu(z)
$$
where~$D^2 \phi$ is the Hessian of~$\phi$. If in addition $D^2 \phi$
is bounded on~$\R^d$, then so is $\Levy^\mu[\phi]$.
\end{enumerate}
\end{remark}

Under~\eqref{levy-measure},~$\Levy^\mu$ is the generator of a pure
jump L\'{e}vy process, and reversely, any pure jump L\'{e}vy process
has a generator of like $\Levy^\mu$ (see e.g.~\cite{Ap:Book,Sato}).
This class of diffusion processes contains e.g. the~$\alpha$-stable
process whose generator is the fractional
Laplacian~$-\left(-\triangle \right)^{\frac{\alpha}{2}}$
with~$\alpha \in (0,2)$. It can be defined for all~$\phi \in
C^\infty_c(\R^d)$ via the Fourier transform as
$$
\left(-\triangle \right)^{\frac{\alpha}{2}} \phi=\mathcal{F}^{-1}
\left( |\cdot|^\alpha \mathcal{F} \phi \right),
$$
or in the form~\eqref{nat:levy-form} with the following L\'evy
measure (see~e.g.~\cite{Ap:Book,Droniou/Imbert,Imb05}):
\begin{equation}\label{frac_lap}
\begin{split}
\dif \mu(z)= \frac{\dif z}{|z|^{d+\alpha}} \mbox{ (up to a positive
multiplicative constant).}
\end{split}
\end{equation}
Many other L\'evy processes/operators of practical interest can be
found in e.g. \cite{Ap:Book,Cont/Tankov}. Under
assumption~\eqref{Aflux},~$\mathcal{L}^\mu [A( \cdot)]$ is an
example of a nonlinear nonlocal diffusion operator. For recent
studies of this and similar type of operators, we refer the reader
to~\cite{BiKaIm10,BiKaMo10,CaVa10,Cifani/Jakobsen,DPQRV} and the
references therein.

Equation~\eqref{1} appears in many different contexts such as
overdriven gas detonations~\cite{Clavin}, mathematical
finance~\cite{Cont/Tankov}, flow in porous media \cite{DPQRV},
radiation hydrodynamics~\cite{RoYo07,Ros89}, and anomalous diffusion
in semiconductor growth \cite{Woyczynski}. Equations of the
form~\eqref{1} constitute a large class of nonlinear degenerate
parabolic integro-differential equations (integro-PDEs). Let us give
some representative examples. 

When $A=0$ or~$\mu = 0$,~\eqref{1} is
the well-known scalar conservation law (see e.g. \cite{Da:Book} and
references therein):
\begin{equation}\label{nat:scl}
\partial_t u+\Div f(u)=0.
\end{equation}
When~$A(u)=u$ and~$\Levy^\mu$ is the fractional Laplacian,~\eqref{1} is the so-called fractal/fractional
conservation law:
\begin{equation}\label{nat:linear-eq}
\partial_t u+\Div f(u)=-\left(-\triangle \right)^{\frac{\alpha}{2}}u.
\end{equation}
Equation~\eqref{nat:linear-eq} has been extensively studied since
the
nineties~\cite{Alibaud,AlAn10,AlDrVo07,AlImKa10,BiFuWo98,BiKaWo99,
BiKaWo01-1,BiKaWo01-2,ChCz10,ChCzSi10,CiJaKa10-1,CiJaKa10-2,DoDuLi09,Droniou,Dro10,
Droniou/Gallouet/Vovelle,Droniou/Imbert,IvMi10,JoMeWo05-1,JoMeWo05-2,JoRo10,KaMiXu08,KiNaSh08,MiYuZh08,MiWu09}.
The case of more general L\'{e}vy diffusions, combined with nonlinear
local diffusions,
 \begin{equation}\label{kensul}
\partial_t u+\Div f(u)=\Div(a(u)\nabla u)+\Levy^\mu[u],
\end{equation}
can be found in~\cite{KaUl10}. 

When~$A$ is nonlinear,~\eqref{1} can be seen as a generalization of
the following classical convection-diffusion equation (possibly
degenerate):
\begin{equation}\label{nat:local-eq}
\partial_t u+\Div f(u)=\triangle A(u);
\end{equation}
see e.g.~\cite{Perthame,Car99,Chen/Karlsen,CoGr99,Karlsen/Risebro} for precise
references on~\eqref{nat:local-eq}. Nonlinear nonlocal diffusions have
been invetigated in \cite{DPQRV} in the setting of nonlocal porous
media equations and $L^1$ semi-group methods, and
in~\cite{Cifani/Jakobsen} where an  $L^\infty\cap L^1$ entropy
solution theory is developed for more general degenerate equations of the
form \eqref{1} along with connections to
Hamilton-Jacobi-Bellman equations of stochastic control theory. 
Other interesting examples concern the class of
nonsingular L\'evy measures satisfying~$ \int_{\R^d \setminus \{0\}}
\dif \mu(z)<+\infty.$ In that case, \eqref{1} can also be seen as a generalization of
Rosenau's models~\cite{KaNi98,KaNi99,LaMa03,LiTa01,ScTa92,Ser03}
and nonlinear radiation hydrodynamics models~\cite{RoYo07} of the
form
\begin{align}
\partial_t u+\Div f(u) & =  g \ast A(u)-A(u),\label{nat:hydro}
\end{align}
where~$\ast$ denotes the convolution product w.r.t.~$x$ and~$g
\in L^1(\R^d)$ is nonnegative with~$\int_{\R^d} g(z) \, \dif z=1$.

Most of the results on these nonlocal equations concern
Equation~\eqref{nat:linear-eq} with the~$\alpha$-stable linear
diffusions, and convolution models~\eqref{nat:hydro} with nonlinear
but nonsingular L\'evy diffusions. It is known that shocks can occur
in finite
time~\cite{AlDrVo07,DoDuLi09,KaNi98,KaNi99,KiNaSh08,LiTa01,ScTa92},
that weak solutions can be nonunique~\cite{AlAn10}, and that the
Cauchy problem is well-posed with the notion of entropy solutions in
the sense of Kruzhkov~\cite{Alibaud,LaMa03,RoYo07}; see
also the works~\cite{CoGrLo96,JaWi04} for the related topic of time fractional
derivatives. The entropy solution theory has been generalized
in~\cite{KaUl10} to singular but linear L\'evy diffusions along with
nonlinear local diffusions. Very recently, it has been extended
in~\cite{Cifani/Jakobsen} to cover the full problem~\eqref{1} for
general singular L\'evy measures and nonlinear $A$.

The purpose of the present paper is to develop an abstract framework
for finding error estimates for entropy solutions of~\eqref{1}. As
applications, we focus  in this paper on continuous dependence
estimates and convergence rates for vanishing viscosity
approximations. We refer the reader
to~\cite{Perthame,Chen/Karlsen,CoGr99,Karlsen/Risebro,Kuznetsov} and the references
therein for similar analysis on~\eqref{nat:local-eq} and related
local equations. As far as nonlocal equations are concerned,
continuous dependence estimates for fully nonlinear integro-PDEs
have already been derived in~\cite{Jakobsen/Karlsen} in the context
of viscosity solutions of Bellman-Isaacs equations; see
also~\cite{Droniou/Imbert,Imb05,Jakobsen/Karlsen} for error
estimates on nonlocal vanishing viscosity approximations. 

To the best of our knowledge, the first and up to now only continuous
dependence estimate 
for nonlocal conservation laws can be found in \cite{KaUl10}; see
also~\cite{Alibaud,CoGrLo96,Droniou,Droniou/Imbert,LaMa03,ScTa92} for
convergence rates for vanishing viscosity approximations of
Equations~\eqref{nat:linear-eq} and~\eqref{nat:hydro}.  The general
estimate in~\cite{KaUl10} is established for Equation~\eqref{kensul} for linear
symmetric L\'evy diffusions. 
Inspired by an early version of the present paper,
a formal discussion on possible extensions to nonlinear nonlocal
diffusions is also given. On the technical side, \cite{KaUl10}
employs so-called entropy defect measures while we do not.

To finish with the bibliography, let us also refer the reader
to~\cite{Cifani/Jakobsen2,CiJaKa10-1,CiJaKa10-2,DeRo04,Dro10,RoYo07}
for the related topic of error estimates for numerical approximations.

Our main result is stated in Lemma~\ref{lem:kuznetsov}, and it
compares the entropy solution~$u$ of~\eqref{1} with a general
function~$v$. Our main application consists in comparing~$u$ with
the entropy solution~$v$ of
\begin{equation}\label{bb1}
\begin{cases}
\partial_t v+\Div g(v)=\Levy^\nu[B(v)],\\
v(x,0)=v_0,
\end{cases}
\end{equation}
where the data set~$(g,B,v_0,\nu) $ is assumed to
satisfy~\eqref{flux}--\eqref{levy-measure}. We obtain explicit
continuous dependence estimates on the data stated in
Theorems~\ref{th:nonlin}--\ref{nat:th-levy}. Let us recall that
when~$B=0$ or~$\nu=0$,~\eqref{bb1} is the pure scalar conservation
law in~\eqref{nat:scl}. Equation~\eqref{1} can thus be seen as a
nonlinear nonlocal vanishing viscosity approximation
of~\eqref{nat:scl} if~$A$ or~$\mu$ vanishes. The rate of convergence
is then obtained as a consequence of
Theorems~\ref{th:nonlin}--\ref{nat:th-levy}, see
Theorem~\ref{nat:cor-rate}.

It is natural to compare Theorems~\ref{th:nonlin}--\ref{nat:th-levy}
and Theorem ~\ref{nat:cor-rate} with the known error estimates for
the different equations above. One can see that a quite important part of
them are particular cases of our general results. We discuss this
point in Section~\ref{nat:sec-discussion} by giving precise
examples. Let us mention that we also give an example of
a simple Hamilton-Jacobi equation where we show that
Theorems~\ref{th:nonlin}--\ref{nat:th-levy} are in some sense the
``conservation law version'' of the results
in~\cite{Jakobsen/Karlsen}; see Example~\ref{nat:ex-HJ}. 

To finish, let us mention that in the case of fractional Laplacians
of order~$\alpha \geq 1$,
Theorems~\ref{th:nonlin}--\ref{nat:th-levy} can be improved by
taking advantage of the homogeneity of the measures
in~\eqref{frac_lap}. In order not to make this paper too long, this
special case (including~$\alpha<1$) is investigated in a second
paper~\cite{AlCiJa10}.

The rest of this paper is organized as follows. In
Section~\ref{nat:sec-preliminaries} 
we 
recall the notion of entropy solution
to~\eqref{1}. In Section~\ref{sec:2}, we state and discuss our main
results.
Sections~~\ref{nat:sec-pro-nl}--\ref{nat:sec-proofs} are devoted to
the proofs of our main results; Section~\ref{nat:sec-pro-nl} states
some preliminary results on the nonlocal operator. 

\subsection*{Notation}

Hereafter,~$a \vee b:=\max \{a,b\}$, while $\cdot$
and~$|\cdot|$ denote the Euclidean inner product and norm. For $A \in \R^{d \times d}$,~$|A|:=\max \{A \, w:w\in \R^d, \, |w|\leq 1\}$. The symbols
$\|\cdot\|$ and $|\cdot|$ are used for norms and semi-norms of functions
respectively.  The symbol~$\mbox{supp}$ is used for the support. The superscripts~$^\pm$ are used for the positive and negative parts. The total variation of a Radon
measure~$\mu$ is denoted by $|\mu|$. Its
tensor product with the Lebesgue measure~$\dif w$ is denoted by~$\dif \mu(z) \,
\dif w$.

\section{Entropy formulation and well-posedness}\label{nat:sec-preliminaries}

Let us recall the formal computations leading to the entropy
formulation of~\eqref{1}.  First we split~$\Levy^\mu$ into 3 parts:
\begin{equation}\label{nat:sum-levy}
\mathcal{L}^\mu[\phi](x)=\Levy_{r}^\mu[\phi](x)+ \Div \left(b_r^\mu
\, \phi \right) (x)+\Levy^{\mu,r}[\phi](x)
\end{equation}
for~$\phi \in C^\infty_c(\R^d)$,~$r>0$, and~$x \in \R^d$, where
\begin{align}
\Levy_{r}^\mu[\phi](x) & :=  \int_{0<|z|\leq r} \phi(x+z)-\phi(x)-z\cdot D\phi(x) \, \mathbf{1}_{|z| \leq 1} \, \dif\mu(z),\label{nat:def-decomposition-inner}\\
b_r^\mu & :=  -\int_{|z| >r} z \mathbf{1}_{|z| \leq 1} \, \dif \mu(z),\label{nat:def-decomposition-b} \\
\Levy^{\mu,r}[\phi](x) & :=  \int_{|z|>r}\phi(x+z)-\phi(x) \,
\dif\mu(z).\label{nat:def-decomposition-outer}
\end{align}
Consider then the Kruzhkov~\cite{Kru70}  entropies~$|\cdot-k|$,~$k
\in \R$, and entropy fluxes
\begin{equation}\label{nat:def-flux}
q_f(u,k):=\sgn(u-k) \, (f(u)-f(k)) \in \R^d,
\end{equation}
where we always use the following everywhere representative of the
sign function:
\begin{equation}\label{nat:representation-sign}
\sgn(u) :=
\begin{cases}
\pm 1& \mbox{ if~$ \pm u > 0$,}\\
0 & \mbox{ if~$u=0$.}
\end{cases}
\end{equation}
By~\eqref{Aflux} it is readily seen that for all~$u,k \in \R$,
\begin{equation}\label{nat:key}
\sgn (u-k) \, (A(u)-A(k)) = |A(u)-A(k)|,
\end{equation}
and we formally deduce from~\eqref{nat:sum-levy},~\eqref{nat:key},
and the nonnegativity of~$\mu$ that
\begin{equation*}
\begin{split}
&\sgn (u-k) \, \Levy^{\mu}[A(u)]\\
&\leq \Levy^{\mu}_r[|A(u)-A(k)|]+\Div \left(b_r^\mu \,
|A(u)-A(k)|\right)+\sgn (u-k) \, \Levy^{\mu,r}[A(u)].
\end{split}
\end{equation*}

Let $u$ be a solution of~\eqref{1}, and multiply~\eqref{1} by~$\sgn
(u-k)$. Formal computations then reveal that
\begin{equation*}
\begin{split}
&\partial_t |u-k|+\Div \, \left(q_f(u,k)- b_r^\mu \, |A(u)-A(k)| \right)\\
&\qquad\qquad\leq\Levy^{\mu}_r[|A(u)-A(k)|]+\sgn (u-k) \,
\Levy^{\mu,r}[A(u)].
\end{split}
\end{equation*}
The entropy formulation in Definition~\ref{L1-entropy} below
consists in asking that~$u$ satisfies this inequality for all
entropy-flux pairs (i.e. for all $k\in\R$) and all~$r>0$. Roughly
speaking one can give a sense to~$\sgn (u-k) \, \Levy^{\mu,r}[A(u)]$
for bounded discontinuous~$u$ thanks to~\eqref{levy-measure}. But
since $\mu$ may be singular at~$z=0$, see Remark \ref{rem1} (3), the
other terms have to be interpreted in the sense of distributions:
Multiply by test functions~$\phi$ and integrate by parts to move
singular operators onto test functions. For the nonlocal terms this
can be done by change of variables: First take~$(z,x,t) \rightarrow
(-z,x,t)$ to see (formally) that
$$
\int_{Q_T} \phi \, \Div \, \left(b_r^\mu \, |A(u)-A(k)|\right)  \dif
x \dif t =\int_{Q_T} D \phi \cdot b_r^{\mu^\ast} \, |A(u)-A(k)| \,
\dif x \dif t,
$$
where
 $\mu^\ast$ is the  L{\'e}vy measure (i.e. it satisfies
 \eqref{levy-measure}) defined for all~$\phi \in C_c^\infty(\R^d \setminus \{0\})$ by
\begin{equation}\label{nat:measure-dual}
\int_{\R^d \setminus \{0\}} \phi(z) \, \dif \mu^\ast(z):=\int_{\R^d \setminus \{0\}} \phi(-z) \, \dif \mu(z).
\end{equation}
In view of \eqref{nat:def-decomposition-inner}, we can take~$(z,x,t)
\rightarrow (-z,x+z,t)$ to find that
$$
\int_{Q_T} \phi \, \Levy^{\mu}_r[|A(u)-A(k)|] \, \dif x \dif t
=\int_{Q_T} |A(u)-A(k)| \, \Levy^{\mu^\ast}_r[\phi] \, \dif x \dif
t.
$$

This leads to the following definition introduced
in~\cite{Cifani/Jakobsen}.
\begin{definition}\label{L1-entropy}\emph{(Entropy solutions)}
Assume~\eqref{flux}--\eqref{levy-measure}. We say that a function~$
u\in L^\infty(Q_T)\cap C \left([0,T];L^1\right) $ is an entropy
solution of~\eqref{1} provided that for all~$k\in\mathbb{R}$,~all
$r>0$, and all nonnegative~$\phi\in C^\infty_c(\R^{d+1})$,
\begin{multline}\label{entropy_ineq}
\int_{Q_{T}} |u-k|\,\partial_t\phi+\left(q_{f}(u,k)+b_r^{\mu^\ast} \, |A(u)-A(k)|\right)\cdot D\phi \, \dif x \dif t\\
+\int_{Q_{T}} |A(u)-A(k)|\,\Levy^{\mu^\ast}_{r}[\phi]+\sgn (u-k)\,\Levy^{\mu,r}[A(u)]\,\phi \, \dif x\, \dif t\\
-\int_{\R^d} |u(x,T)-k|\,\phi(x,T) \, \dif x +\int_{\R^d}
|u_0(x)-k|\,\phi(x,0) \, \dif x\geq 0.
\end{multline}
\end{definition}

\begin{remark}\label{modif:rem}\

\begin{enumerate}
\item \label{rem-ref-l} 
Under assumptions \eqref{flux}--\eqref{levy-measure}, the entropy
inequality~\eqref{entropy_ineq} is well-defined independently of the
a.e.~representative of $u$. To see this, note that since $\mu^\ast$
satisfies~\eqref{levy-measure}, it easily
follows that $\Levy^{\mu^\ast}_r[\phi]\in C_c^\infty(\R^{d+1})$.
Since~$\sgn(u-k)$,~$q_f(u,k)$, and~$A(u)$ belong to~$L^\infty$
by~\eqref{nat:representation-sign} and~\eqref{flux}--\eqref{Aflux},
it is then clear that all terms in~\eqref{entropy_ineq} are
well-defined except possibly the $\Levy^{\mu,r}$-term. Here  it may
look like we are integrating Lebesgue measurable functions w.r.t.~a
Radon measure $\mu$. 
However, the integrand does have the right
  measurability by a classical approximation procedure, see Remark 5.1
  in \cite{Cifani/Jakobsen}.  We 
therefore find that since~$A(u)$ belongs to~$C([0,T];L^1)$, so does
also~$\Levy^{\mu,r}[A(u)]$ and we are done.
\item \label{rem-ref} Another way to understand the measurability issue in \eqref{rem-ref-l},
is simply to consider only Borel measurable a.e. representatives of the solutions. The reading of the paper would remain exactly the same, since our~$L^1$-continuous dependence estimate do not depend on the representatives.
\item In the definition of entropy solutions, it is possible to
  consider functions~$u$ only defined for a.e. $t\in[0,T]$ by taking
  test functions with compact support in $Q_T$ and adding an explicit
  initial condition, see e.g. \cite{Cifani/Jakobsen}.
\item One can check that classical solutions are entropy solutions,
  thus justifying the formal computations leading to
  Definition \ref{L1-entropy}. Moreover entropy
  solution are weak solutions and hence smooth entropy solutions are
  classical solutions. We refer the reader to~\cite{Cifani/Jakobsen}
  for the proofs.
\end{enumerate}
\end{remark}

Here is a well-posedness result from~\cite{Cifani/Jakobsen}.
\begin{theorem}\emph{(Well-posedness)}\label{th:well-posedness}
Assume~\eqref{flux}--\eqref{levy-measure}. There exists a unique
entropy solution~$u$ of~\eqref{1}. This entropy solution belongs
to~$L^\infty(Q_T)\cap C \left([0,T];L^1\right) \cap L^\infty
\left(0,T;BV\right) $ and
\begin{equation}\label{nat:nonincrease}
\begin{cases}
\|u\|_{L^\infty(Q_T)}\leq\|u_0\|_{L^\infty(\mathbb{R}^d)},\\
\|u\|_{C \left([0,T];L^1 \right)}\leq\|u_0\|_{L^1(\mathbb{R}^d)},\\
|u|_{L^\infty \left(0,T;BV\right)}\leq|u_0|_{BV(\mathbb{R}^d)}.
\end{cases}
\end{equation}
Moreover, if $v$ is the entropy solution of~\eqref{1}
with~$v(0)=v_0$ for another initial data~$v_0$
satisfying~\eqref{in_condition}, then
\begin{equation}\label{nat:L1-contraction}
\|u-v\|_{C([0,T];L^1)} \leq \|u_0-v_0\|_{L^1(\R^d)}.
\end{equation}
\end{theorem}

\begin{remark}\label{nat:rem-loc-lip}
By the~$L^\infty$-estimate in~\eqref{nat:nonincrease}, all the
results of this paper also holds for locally Lipschitz-continuous
nonlinearities $(f,A)$. Simply replace the data $(f,A)$ by $(f,A) \,
\psi_M$, where $\psi_M\in
C_c^\infty(\R)$ is such that $\psi_M=1$ in $[-M,M]$ for
$M=\|u_0\|_{L^\infty(\R^d)}$. 
\end{remark}

\section{Main results}\label{sec:2}

Our first main result  is a Kuznetsov type of lemma that measures
the distance between the entropy solution~$u$ of~\eqref{1} and an
arbitrary function~$v$.

Let~$\epsilon, \delta>0$ and $\phi^{\epsilon,\delta} \in
C^\infty(Q_T^2)$ be the test function
\begin{eqnarray}
\phi^{\epsilon,\delta}(x,t,y,s) := \theta_\delta(t-s) \,
\bar{\theta}_{\epsilon}(x-y),\label{nat:test-kuznetsov}
\end{eqnarray}
where~$\theta_\delta(t) := \frac{1}{\delta} \,
\tilde\theta_1\big(\frac{t}{\delta} \big)$
and~$\bar{\theta}_{\epsilon}(x) :=
\frac{1}{\epsilon^d}\tilde\theta_d\big(\frac{x}{\epsilon}\big)$ are,
respectively, time and space approximate units with kernel
$\tilde\theta_{n } $ with $n=1$ and $n=d$ satisfying
\begin{equation}\label{nat:smooth-kernel}
\tilde\theta_{n } \in C_{c}^{\infty}(\mathbb{R}^{n }),\quad
\tilde\theta_{n } \geq 0,\quad \mbox{supp} \, \tilde \theta_n
\subseteq \{|x|<1\} , \quad\text{and}\quad \int_{\R^{n }}
\tilde\theta_{n } (x ) \, \dif x =1.
\end{equation}
We also let $\omega_u(\delta)$ be the modulus of continuity
of~$u \in C\left([0,T];L^1\right)$.

\begin{lemma}[Kuznetsov type Lemma]\label{lem:kuznetsov}
Assume~\eqref{flux}--\eqref{levy-measure}. Let~$u$ be the entropy
solution of~\eqref{1} and~$v\in L^\infty(Q_T)\cap C
\left([0,T];L^1\right)$ with~$v(0)=v_0$. Then for all~$r>0$,
$\epsilon>0$, and $0<\delta<T$,
\begin{equation}\label{kuz}
\begin{split}
& \|u(T)-v(T)\|_{L^{1}(\mathbb{R}^d)}\\
& \leq\|u_0-v_0\|_{L^{1}(\mathbb{R}^d)}+\epsilon \, C_{\tilde\theta} \, |u_0|_{BV(\R^d)}+2\,\omega_{u}(\delta) \vee  \omega_{v}(\delta) \\
& \quad -\iint_{Q_{T}^2}|v(x,t)-u(y,s)|\,\partial_t \phi^{\epsilon,\delta}(x,t,y,s)\, \dif w\\
& \quad  -\iint_{Q_{T}^2} \left(q_f(v(x,t),u(y,s))+b_r^{\mu^\ast} \, |A(v(x,t))-A(u(y,s))|\right) \cdot D_x\phi^{\epsilon,\delta}(x,t,y,s)\, \dif w\\
&\quad+\iint_{Q_{T}^2}
|A(v(x,t))-A(u(y,s))|\,\Levy^{\mu^\ast}_{r}[\phi^{\epsilon,\delta}(x,t,\cdot,s)](y)
\, \dif w \\
& \quad  -\iint_{Q_{T}^2} \sgn (v(x,t)-u(y,s))\,\Levy^{\mu,r}[A(u(\cdot,s))](y) \,\phi^{\epsilon,\delta}(x,t,y,s) \, \dif w\\
& \quad +\iint_{\R^d \times Q_T} |v(x,T)-u(y,s)|\,\phi^{\epsilon,\delta}(x,T,y,s) \, \dif x \, \dif y \, \dif s\\
& \quad -\iint_{\R^d \times Q_T}
|v_0(x)-u(y,s)|\,\phi^{\epsilon,\delta}(x,0,y,s) \, \dif x \, \dif y
\, \dif s
\end{split}
\end{equation}
where~$\dif w:=\dif x\, \dif t \, \dif y\,\dif s$, and
$C_{\tilde\theta}:=2 \int_{\R^d} |x| \tilde\theta_d(x) \, \dif x$.
\end{lemma}

\begin{remark}\*
\begin{enumerate}
\item The error in time only depends on the moduli of continuity
  of~$u$ and~$v$ at~$t=0$ and~$t=T$.
Here we simply take the global-in-time moduli of
continuity~$\omega_u(\delta)$ and~$\omega_v(\delta)$, since this is
sufficient in our settings.
\item
When $A =  0$ or $\mu = 0$ this lemma reduces to the well-known
Kuznetsov lemma \cite{Kuznetsov} for multidimensional scalar
conservation laws.
\item Notice that the~$\Levy^{\mu^\ast}_r$-term vanishes when $r \rightarrow 0$, see Lemma~\ref{Kuz:lem}.
\item Lemma
\ref{lem:kuznetsov} has many applications. In this paper  and in
  \cite{AlCiJa10}
  we focus on continuous dependence results and error estimates for
  the vanishing viscosity method. Then in \cite{Cifani/Jakobsen2}, we
  will use the lemma to obtain error estimates for numerical
  approximations of \eqref{1}. 
\end{enumerate}
\end{remark}

In this paper we apply Lemma~\ref{lem:kuznetsov} to compare the
entropy solution~$u$ of~\eqref{1} with the entropy solution~$v$
of~\eqref{bb1}. This is our second main result, and we present it
in the two theorems below. The first focuses on the dependence on
the nonlinearities (with~$\mu=\nu$) and the second one on the L\'evy
measure (with~$A=B$).

\begin{theorem}\emph{(Continuous dependence on the nonlinearities)}\label{th:nonlin}
Let~$u$ and~$v$ be the entropy solutions of~\eqref{1}
and~\eqref{bb1} respectively with data sets~$(f,A,u_0,\mu) $ and
$(g,B,v_0,\nu=\mu) $ satisfying~\eqref{flux}--\eqref{levy-measure}.
Then for all $T,r>0$,
\begin{equation}\label{nat:first-app}
\begin{split}
& \|u-v\|_{C\left([0,T];L^{1}\right)} \leq\|u_0-v_0\|_{L^{1}(\mathbb{R}^d)}+|u_0|_{BV(\R^d)} \, T \, \|f'-g'\|_{L^\infty(\R,\R^d)}\\
& \quad + |u_0|_{BV(\R^d)} \,\sqrt{c_d \,  T  \int_{0<|z|\leq r}
|z|^2 \, \dif \mu(z)
 \, \|A'-B'\|_{L^\infty(\R)}}\\
& \quad + |u_0|_{BV(\R^d)} \, T\, \left|\int_{r \wedge 1 < |z|
\leq r \vee 1} z \, \dif \mu(z)\right|
 \,
\|A'-B'\|_{L^\infty(\R)}\\
& \quad +  T \, \int_{|z|> r} \|u_0(\cdot+z)-u_0
\|_{L^1(\R^d)} \, \dif \mu(z) \, \|A'-B'\|_{L^\infty(\R)},
\end{split}
\end{equation}
where $c_d=\frac{4d^2}{d+1}$.
\end{theorem}

\begin{theorem}\emph{(Continuous dependence on the L\'evy measure)}\label{nat:th-levy}
Let~$u$ and~$v$ be the  entropy solutions of~\eqref{1}
and~\eqref{bb1} respectively with data sets~$(f,A,u_0,\mu) $
and~$(g,B=A,v_0,\nu) $
satisfying~\eqref{flux}--\eqref{levy-measure}. Then for
all~$T,r>0$,
\begin{equation}\label{nat:second-app}
\begin{split}
& \|u-v\|_{C\left([0,T];L^{1}\right)} \leq\|u_0-v_0\|_{L^{1}(\mathbb{R}^d)}+|u_0|_{BV(\R^d)} \, T \,  \|f'-g'\|_{L^{\infty}(\R,\R^d)}\\
& \quad + |u_0|_{BV(\R^d)} \,\sqrt{c_d \,  T\, \|A'\|_{L^\infty(\R)}  \int_{0<|z|\leq r} |z|^2 \,\dif |\mu-\nu|(z)}\\
& \quad + |u_0|_{BV(\R^d)} \, T\, \|A'\|_{L^\infty(\R)} \,
\left|\int_{r \wedge 1 < |z| \leq r \vee 1} z \, \dif
(\mu-\nu)(z)\right|
\\
& \quad +  T \, \|A'\|_{L^\infty(\R)} \, \int_{|z|> r}
\|u_0(\cdot+z)-u_0 \|_{L^1(\R^d)} \, \dif |\mu-\nu|(z),
\end{split}
\end{equation}
where $c_d=\frac{4d^2}{d+1}$.
\end{theorem}

\begin{remark}
In the error estimates of Theorems~\ref{th:nonlin} and~\ref{nat:th-levy}, there are 3 terms accounting for the dependence on the 
fractional diffusion term in~\eqref{1}: One term accounts for the
behavior near the singularity of $\mu$ at $z=0$ (the
integral over $0<|z| \leq r$), another term accounts
for the behavior near infinity (the integral over $|z| \geq r$), and
the last term  (the integral over $r \wedge 1 < |z| \leq r \vee 1$)
is a drift term that is only present for nonsymmetric measures
$\mu$. The square root estimate for the singular term is similar to
 estimates for 2nd derivative terms in the local case and
for non-local equations with different structure,
cf. e.g. \cite{Chen/Karlsen,Jakobsen/Karlsen,KaUl10}.
\end{remark}

\begin{remark}\label{rem:splitting}
Since the initial data is $L^1 \cap BV$, an application of Fubini's
theorem shows that for any  $\hat r>r>0$,
\begin{equation*}
\begin{split}
&\int_{|z|>r}
\|u_0(\cdot+z)-u_0 \|_{L^1(\R^d)} \, \dif \mu(z)\\
&\leq |u_0|_{BV(\R^d)}\int_{r<|z| \leq \hat r} |z| \, \dif\mu(z)
+2\|u_0\|_{L^1(\R^d)} \int_{|z| > \hat r} \dif\mu(z).
\end{split}
\end{equation*}
\end{remark}

 From Theorems \ref{th:nonlin} and \ref{nat:th-levy}  we can easily
find a general continuous dependence estimate
  when both $A$ and $\mu$ are different from $B$ and $\nu$,
  respectively. E.g. we can take an intermediate
  solution $w$ of $w_t+\Div \, f(w)=\Levy^\mu[B(w)]$ and $w(0)=u_0$,
  and use
  the triangle inequality. Using this idea we can show that
  the following estimates always have to hold:
\begin{corollary}\label{cor_main}
Let~$u$ and~$v$ be the  entropy solutions of \eqref{1} and
\eqref{bb1} respectively with data sets $(f,A,u_0,\mu) $ and
$(g,B,v_0,\nu) $ satisfying~\eqref{flux}--\eqref{levy-measure}. Then for all~$T>0$
\begin{align}
\label{main1}
\begin{split}
&\|u-v\|_{C\left([0,T];L^{1}\right)}\leq
\|u_0-v_0\|_{L^{1}(\mathbb{R}^d)}+|u_0|_{BV(\R^d)} \, T \,
\|f'-g'\|_{L^{\infty}(\R,\R^d)}\\
&\quad +  C\, (T^{\frac12}\vee T) \,\bigg(\sqrt{
\|A'-B'\|_{L^\infty(\R)}} +\sqrt{ \int_{\R^d \setminus \{0\}}
|z|^2\wedge1 \,\dif
  |\mu-\nu|(z)}\bigg)
\end{split}
\end{align}
where $C$ only depends on $d$ and the data.
Moreover, if in addition
$$\int_{\R^d \setminus \{0\}}|z|\wedge 1\,\dif
 \mu(z)+\int_{\R^d \setminus \{0\}}|z|\wedge 1\,\dif \nu(z)<+\infty,$$
then we have the better estimate
\begin{align}
\label{main2}
\begin{split}
&\|u-v\|_{C\left([0,T];L^{1}\right)} \leq
\|u_0-v_0\|_{L^{1}(\mathbb{R}^d)}+|u_0|_{BV(\R^d)} \, T \,
\|f'-g'\|_{L^{\infty}(\R,\R^d)}\\
&\quad +  C T \,\bigg(\|A'-B'\|_{L^\infty(\R)} + \int_{\R^d
\setminus \{0\}} |z|\wedge1 \,\dif |\mu-\nu|(z)\bigg),
\end{split}
\end{align}
where $C$ only depends on the data.
\end{corollary}
\begin{proof}[Outline of proof]
To prove \eqref{main1}, we use Theorems  \ref{th:nonlin} and
\ref{nat:th-levy} with $r=1$ and the triangle inequality. We
also use estimates like $|a-b|\leq \sqrt{|a|+|b|}\sqrt{|a-b|}$,
 $|\mu-\nu|\leq|\mu|+|\nu|$ etc. To prove \eqref{main2}, we also use
 Remark \ref{rem:splitting} and set $r=0$ and $\hat r=1$.
\end{proof}

\begin{remark}\*
\begin{enumerate}
\item All these estimates hold for arbitrary L{\'e}vy measures $\mu,\nu$ and even
  for strongly degenerate diffusions where $A,B$ may vanish on large
  sets. They are consistent (at least for the $|\mu-\nu|$ term)
with general results for nonlocal Hamilton-Jacobi-Bellman equations
in \cite{Jakobsen/Karlsen}. When $\mu,\nu$ have the special form
\eqref{frac_lap} (with possibly different $\alp$'s), then it is
possible to use the extra symmetry and homogeneity properties to
obtain better estimates, see \cite{AlCiJa10}.
\item The optimal choice of the $r,\hat r$ in Remark \ref{rem:splitting} depends on the
  behavior of the L\'evy measures at zero and infinity, see the
  discussion above and at the end of this section for more details.
\end{enumerate}
\end{remark}

Let us now consider the nonlocal vanishing viscosity problem
\begin{equation}\label{nat:bb2}
\begin{cases}
\partial_t u^\epsilon+\Div f(u^\epsilon)=\epsilon \, \Levy^\mu[A(u^\epsilon)],\\
u^\epsilon(0)=u_0,
\end{cases}
\end{equation}
i.e.~problem \eqref{nat:scl} with a perturbation term $\epsilon \,
\Levy^\mu[A(u^\epsilon)]$. When $\epsilon>0$ tend to zero,
$u^\epsilon$ is expected to converge toward the solution $u$ of
\eqref{nat:scl}. As an immediate application of Theorem
\ref{th:nonlin} or \ref{nat:th-levy}, we have the following  result:
\begin{theorem}[Vanishing viscosity]\label{nat:cor-rate}
Assume~\eqref{flux}--\eqref{levy-measure}. Let~$u$ and~$u^\epsilon$
be the  entropy solutions of~\eqref{nat:scl} and~\eqref{nat:bb2}
respectively. Then for every~$T,\epsilon>0$ and all~$\hat r>r>0$,
\begin{equation}\label{nat:rate}
\begin{split}
&\|u-u^\epsilon\|_{C([0,T];L^1)}\leq C \min_{\hat r>r>0} \left\{d^{\frac12}  T^{\frac12} \epsilon^{\frac12} \sqrt{\int_{0<|z| \leq r} |z|^2 \, \dif \mu(z)} \right. \\
&\quad+\left.T\epsilon \,  \bigg[\int_{r<|z| \leq \hat r} |z| \, \dif
    \mu(z)
+  \Big|\int_{r \wedge 1 < |z| \leq r \vee 1} z \, \dif
\mu(z)\Big| + \int_{|z| > \hat r} \, \dif \mu(z)\bigg]\right\},
\end{split}
\end{equation}
where $C$ only depends on $\|u_0\|_{L^1(\R^d)
  \cap BV(\R^d)}$ and $\|A'\|_{L^\infty(\R)}$.
\end{theorem}

\begin{proof}[Outline of proof]
Note that $u$ can be seen as the entropy solution of \eqref{1} with
$A=0$ and $\mu$ as L\'evy measure. Hence we can estimate
$\|u-u^\epsilon\|_{C([0,T];L^1)}$ from Theorem~\ref{th:nonlin}. The
error coming from the difference of the derivatives of the
nonlinearities is equal to $\epsilon \, \|A'\|_{L^\infty(\R)}$.
Inequality \eqref{nat:rate} then follows from \eqref{nat:first-app}
and Remark \ref{rem:splitting}.
\end{proof}
\begin{corollary}
\label{cor-rate} Assume~\eqref{flux}--\eqref{levy-measure}. Let~$u$
and $u^\epsilon$ be the entropy solutions of \eqref{nat:scl}
and \eqref{nat:bb2} respectively. Then for all~$T,\epsilon>0$
\begin{equation*} 
\|u-u^\epsilon\|_{C([0,T];L^1)}\leq C\, (T^{\frac12}\vee
T)\,\epsilon^{\frac12},
\end{equation*}
where $C$ only depends on  $ d$ and the data.
Moreover, if in addition
\begin{equation}\label{add-referee}
\int_{\R^d \setminus \{0\}}|z|\wedge 1\,\dif
 \mu(z)<+\infty,
\end{equation} 
then we have the better estimate
\begin{equation*} 
\|u-u^\epsilon\|_{C([0,T];L^1)}\leq CT\epsilon,
\end{equation*}
where $C$ depends on the data.
\end{corollary}

This corollary follows immediately from Theorem \ref{nat:cor-rate}
or Corollary \ref{cor_main}.

\begin{remark}\
\begin{enumerate}
\item Our estimates are just as good or better than the standard $\mathcal
  O(\epsilon
  ^{\frac12})$ estimate for
  the classical vanishing viscosity method (\eqref{nat:local-eq} with
  $A(u)=\epsilon
  \, u$).
\item Our estimates hold for arbitrary L{\'e}vy measures $\mu$ and even
  for strongly degenerate diffusions where $A$ may vanish on a large
  set! This is consistent with general results for nonlocal
  Hamilton-Jacobi-Bellman equations \cite{Jakobsen/Karlsen}.
\item As for the classical (local) vanishing viscosity method,
  better rates could be obtained if the solutions are more
  regular. E.g. if~$A(u^\epsilon)$ is uniformly (in $\epsilon$) bounded in
  $W^{2,1}$, then the error estimate should be $O(\epsilon)$ even 
without assumption~\eqref{add-referee}. Such a result can
not be derived from~\eqref{nat:rate}, but should be proved directly.
\item Corollary \ref{cor-rate} contains
less information than Theorem
  \ref{nat:cor-rate}; indeed, if $\mu$ is as in \eqref{frac_lap}, the 
additional symmetry and homogeneity can be used to obtain better
estimates which can be proved to be optimal. See Example
\ref{rate-alp} below. 
\item The error estimates above trivially
  also holds for  the more general vanishing viscosity equation
\begin{equation*}
\begin{cases}
\partial_t u^\epsilon+\Div f(u^\epsilon)=\Levy^\nu[B(u^\epsilon)]+\epsilon \, \Levy^\mu[A(u^\epsilon)],\\
u^\epsilon(0)=u_0.
\end{cases}
\end{equation*}
\end{enumerate}
\end{remark}

\subsection*{Further discussion}\label{nat:sec-discussion}

We now make a more precise comparison of the results above with
known estimates from the literature. We begin with continuous
dependence estimates and finish with convergence rates for vanishing
viscosity approximations.

Let $u$ and~$v$ denote the entropy solutions of~\eqref{1}
and~\eqref{bb1}, respectively. To simplify, we take the same data
sets~$(f,A,u_0)=(g,B,v_0)$ and we only allow the L\'evy
measures~$\mu$ and~$\nu$ to be different. We also let~$C$ denote a
constant only depending on~$T,d $ and the data.

\begin{example}
Let us consider Equation~\eqref{kensul} with~$a=0$. Let
us also consider the class of L\'evy operators satisfying
\begin{equation*}
\begin{cases}
\int_{\R^d \setminus \{0\}} |z|^2 \wedge |z| \, \dif \mu(z)<+\infty,\\
\mu=\mu^\ast.
\end{cases}
\end{equation*}
For such kind of equations, the following continuous dependence
estimate on the L\'evy measure has been established
in~\cite{KaUl10}:
\begin{equation*}
\begin{split}
\|u-v\|_{C\left([0,T];L^{1}\right)} \leq C \sqrt{\int_{0<|z| \leq 1} |z|^2
\, \dif |\mu-\nu|(z)}+C\int_{|z|>1} |z| \, \dif |\mu-\nu|(z).
\end{split}
\end{equation*}
This estimate follows from Theorem \ref{nat:th-levy} and Remark \ref{rem:splitting} by taking
$r=1$ and $\hat r=+\infty$ in~\eqref{nat:second-app}.
\end{example}

\begin{example}\label{nat:ex-HJ}
Consider the following one-dimensional Hamilton-Jacobi
equation
\begin{equation*}
U_t+f(U_x)=\Levy^\mu[U]
\end{equation*}
with initial data $U_0(x):=\int_{-\infty}^x u_0(y) \, \dif y$. This
particular equation is related to the nonlocal conservation law
\eqref{nat:scl}, since its solution is $U(x,t)=\int_{-\infty}^x u(y,t) \,
\dif y$ where $u$ solves \eqref{nat:scl}, see
\cite{Cifani/Jakobsen}.
It is also an example of an integro-PDE for which the general theory
of \cite{Jakobsen/Karlsen} applies, and this theory allows us to
establish the following continuous dependence estimate on the L\'evy
measure:
$$
\sup_{\R \times [0,T]} |U-V| \leq C\sqrt{\int_{\R \setminus \{0\}}
|z|^2\wedge 1 \, \dif |\mu-\nu|(z)},
$$
where~$V(x,t):=\int_{-\infty}^x v(y,t) \, \dif y$. (This result is a
version of Theorem 4.1 in \cite{Jakobsen/Karlsen} which follows
from Theorem 3.1  by setting $p_0,\dots,p_4,p_s=0$ and
$\rho=|z|\wedge1$ in (A0)). Since
$$\sup_{\R \times [0,T]} |U-V| \leq \|u-v\|_{C([0,T];L^1)},$$
this estimate also follows from \eqref{main1} in Corollary
\ref{cor_main} when $(A,f,u_0)=(B,g,v_0)$.
\end{example}

Let us now compare Theorem ~\ref{nat:cor-rate} with known
convergence rates. We keep the same notation for~$u$ and~$u^\epsilon$
as in Theorem ~\ref{nat:cor-rate}.

\begin{example}
\label{rate-alp} Let us consider the case
where~$A(u^\epsilon)=u^\epsilon$
and~$\Levy^\mu=-(-\triangle)^{\frac{\alpha}{2}}$, $\alpha \in
(0,2)$. Then the following optimal rates have been derived in
\cite{Alibaud,Droniou}:
\begin{equation}\label{nat:rate-fl}
\|u-u^\epsilon\|_{C([0,T];L^1)}=
\begin{cases}
\mathcal{O} \left(\epsilon^{\frac{1}{\alpha}} \right) & \mbox{ if~$\alpha>1$},\\
\mathcal{O} \left(\epsilon \, |\ln \epsilon| \right) & \mbox{ if~$\alpha=1$},\\
\mathcal{O} \left(\epsilon \right) & \mbox{ if~$\alpha<1$}.\\
\end{cases}
\end{equation}
Let us explain how these results can be deduced from
\eqref{nat:rate}. First we use \eqref{frac_lap} to explicitly
compute the integrals in \eqref{nat:rate} and obtain
$$
\|u-u^\epsilon\|_{C([0,T];L^1)}= \mathcal{O} \left( \min_{\hat r>r>0}
\left\{ \sqrt{\epsilon \, \frac{r^{2-\alpha}}{2-\alpha}} +
\epsilon \int_{r}^{\hat r} \frac{\dif \tau}{\tau^\alpha} + \epsilon
\, \hat r^{-\alpha} \right\} \right).
$$
We then deduce \eqref{nat:rate-fl} by taking
$r=\epsilon^{\frac{1}{\alpha}}$ and~$\hat r=+\infty$
if~$\alpha>1$,~$r=\epsilon$ and~$\hat r=1$ if~$\alpha=1$, and~$r=0$
and~$\hat r=1$ if~$\alpha<1$.
\end{example}


\begin{example} Let us finally consider the vanishing approximation~\eqref{nat:bb2} with the viscous term 
$$
\partial_t u^\epsilon+\Div f(u^\epsilon)=\frac{1}{\epsilon} \left(g_\epsilon \ast u^\epsilon-u^\epsilon \right),
$$
where~$g_\epsilon(z):=\frac{1}{\epsilon^d} \, g \left( \frac{z}{\epsilon} \right)$ with an even and nonnegative kernel $g \in L^1(\R^d)$ such that 
$$
\int_{\R^d} |z|^2 \,  g(z) \, \dif z<+\infty.
$$ 
This is the Rosenau's regularization of the Chapman-Enskog expansion for hydrodynamics~\cite{Ros89}; see also Equations (1.1) and (2.3) of \cite{ScTa92}. Its convergence toward~\eqref{nat:scl} has been established in~\cite{LaMa03,ScTa92}. In Corollary 5.2 of \cite{ScTa92} the following rate of convergence has been derived:
$$
\|u-u^\epsilon\|_{C([0,T];L^1)}=\mathcal{O} \left(\epsilon^\frac{1}{2}\right).
$$
This result can be recovered from Theorems \ref{th:nonlin} or \ref{nat:th-levy}. Indeed, we can choose e.g. $A(u^\epsilon)=u^\epsilon$, $\dif \mu(z)=\frac{g_\epsilon(z)}{\epsilon} \, \dif z$ and $\nu=0$, to get the desired equations. Next, we apply \eqref{nat:second-app} with~$r=+\infty$ and rescale the~$z$-variable to show that the error term is bounded above by
\begin{equation*}
\begin{split}
C \,\sqrt{\int_{\R^d \setminus \{0\}} |z|^2 \,\dif |\mu-\nu|(z)} 
 & =C \, \sqrt{\int_{\R^d} |z|^2 \, \frac{g\left(\frac{z}{\epsilon}\right)}{\epsilon^{d+1}} \, \dif z} \\
&  = C \,  \epsilon^\frac{1}{2} \, \sqrt{\int_{\R^d} |z|^2 \, g(z) \, \dif z.}
\end{split}
\end{equation*}
\end{example}

\section{Auxiliary results}\label{nat:sec-pro-nl}

Before proving our main results in the next section, we state
several technical lemmas.

\begin{lemma}\label{nat:def-singular}
Assume~\eqref{levy-measure} and $r >0$. Then for all $\phi \in
C^\infty_c(\R^d)$,
$$
\|\Levy^\mu_r[\phi]\|_{L^1(\R^d)}  \leq  \int_{0<|z|\leq r} |z|^2 \,
\dif \mu(z) \, \|\phi\|_{W^{2,1}(\R^d)}.$$
\end{lemma}

The proof easily follows from a Taylor expansion and Fubini's theorem. 
\begin{remark}\label{rem:splitting-b}
Similarly~$\Levy^{\mu,r}$ is (linear and) bounded from~$L^1$ into itself, thanks to Remark~\ref{rem:splitting} with $\hat r=r$. 
\end{remark}
In the next result, we establish a Kato type inequality for
$\Levy^{\mu,r}[A(u)] $.
\begin{lemma}\label{nat:ipp}
Assume~\eqref{Aflux} and \eqref{levy-measure}. Then for all~$u \in
L^1(\R^d)$,~$k \in \R$,~$r>0$, and all $0\leq\phi \in C^\infty_c(\R^d)$,
\begin{equation*}
\int_{\R^d} \sgn (u-k) \, \Levy^{\mu,r}[A(u)] \, \phi \, \dif x \leq
\int_{\R^d} |A(u)-A(k)| \, \Levy^{\mu^\ast,r}[\phi] \, \dif x.
\end{equation*}
\end{lemma}

\begin{proof}
Note first that~$A(u)$ is~$L^1$ by~\eqref{Aflux}, and hence $
\Levy^{\mu,r}[A(u)]$ is well-defined in~$L^1$ by
Remark~\ref{rem:splitting-b}. Easy computations then reveal that
\begin{align*}
& \int_{\R^d} \sgn (u-k) \, \Levy^{\mu,r}[A(u)] \, \phi \, \dif x,\\
& =  \int_{\R^d} \int_{|z|>r}  \sgn (u(x)-k) \, \Big(A(u(x+z))-A(u(x))\Big) \, \phi(x)  \, \dif \mu(z) \, \dif x,\\
& =  \int_{\R^d} \int_{|z|>r} \sgn (u(x)-k)\\
& \qquad\qquad\qquad\qquad \, \Big\{(A(u(x+z))-A(k))-(A(u(x))-A(k))\Big\} \, \phi(x)  \, \dif \mu(z) \, \dif x,\\
& \leq  \int_{\R^d} \int_{|z|>r} \Big(|A(u(x+z))-A(k)|-|A(u(x))-A(k)|\Big) \, \phi(x)  \, \dif \mu(z) \, \dif x \quad \mbox{by~\eqref{nat:key}},\\
& =  \underbrace{\int_{\R^d} \int_{|z|>r}  |A(u(x+z))-A(k)| \, \phi(x)  \, \dif \mu(z) \, \dif x}_{=:I}\\
& \qquad\qquad\qquad - \underbrace{\int_{\R^d} \int_{|z|>r}
|A(u(x))-A(k)| \, \phi(x)  \, \dif \mu(z) \, \dif x}_{=:J}.
\end{align*}
Note that all these integrals are well-defined, thanks
to~\eqref{levy-measure} (\footnote{The measurability is immediate if
  the reader only consider Borel measurable representatives of~$u$ as
  suggested in Remark~\ref{modif:rem}~\eqref{rem-ref}.}).

By the respective changes of variable~$(z,x) \rightarrow (-z,x+z)$
and~$(z,x) \rightarrow (-z,x)$, we find that
\begin{align*}
& I
=  \int_{\R^d} \int_{|z|>r} \phi(x+z) \, |A(u(x))-A(k)|  \, \dif
\mu^\ast(z) \, \dif x,
\\
& J= \int_{\R^d} \int_{|z|>r} \phi(x) \, |A(u(x))-A(k)|  \, \dif
\mu^\ast(z) \, \dif x.
\end{align*}
Here the measure~$\mu^\ast$ in \eqref{nat:measure-dual} appears because
of the relabelling of~$z$. This measure has the same properties as
$\mu$. Hence we can conclude that
$$\int_{\R^d} \sgn (u-k) \, \Levy^{\mu,r}[A(u)] \, \phi \, \dif x \leq I-J=\int_{\R^d} |A(u)-A(k)| \, \Levy^{\mu^\ast,r}[\phi] \, \dif x,
$$
and the proposition follows.
\end{proof}

The next lemma is a consequence of the Kato inequality, and it plays
a key role in the doubling of variables arguments throughout this
paper and  in the uniqueness proof of
\cite{Alibaud,Cifani/Jakobsen}.   
\begin{lemma}\label{nat:pro-kato-d}
Assume~\eqref{Aflux} and \eqref{levy-measure}, and let $ u,v \in
L^\infty(Q_T) \cap C([0,T];L^1)$, $0\leq\psi\in L^1(\R^{d}\times
(0,T)^2) $, and~$r>0$. Then
\begin{align*}
\iint_{Q_{T}^2} &\sgn
(u(y,s)-v(x,t))\\
&\cdot\Big(\Levy^{\mu,r}[A(u(\cdot,s))](y)-\Levy^{\mu,r}[A(v(\cdot,t))](x)\Big)
\,\psi(x-y,t,s) \, \dif w\leq 0
\end{align*}
(where~$\dif w=\dif x \, \dif t \, \dif y \, \dif s$).
\end{lemma}

\begin{proof} 
Note that
\begin{align*}
&\sgn (u(y,s)-v(x,t)) \, \Big( A(u(y+z,s))-A(u(y,s)) \Big)\\
&\qquad-\sgn (u(y,s)-v(x,t)) \, \Big( A(v(x+z,t))-A(v(x,t)) \Big)\\
& = \sgn (u(y,s)-v(x,t))\\
&\qquad\cdot\Big\{ \Big(A(u(y+z,s))-A(v(x+z,t))\Big)-\Big(A(u(y,s))-A(v(x,t))\Big) \Big\}\\
& \leq
\left|A(u(y+z,s))-A(v(x+z,t))\right|-\left|A(u(y,s))-A(v(x,t))\right|
\end{align*}
where these functions are both defined. By 
an integration w.r.t.
$\mathbf{1}_{|z|>r} \, \dif \mu(z)$, we find that for all $(t,s) \in
(0,T)^2$ and a.e.~$(x,y) \in \R^{2d}$,
\begin{align*}
& \sgn (u(y,s)-v(x,t)) \, \Big( \Levy^{\mu,r}[A(u(\cdot,s))](y)-\Levy^{\mu,r}[A(v(\cdot,t))](x) \Big) \\
& \leq  \int_{|z|>r}
\left|A(u(y+z,s))-A(v(x+z,t))\right|-\left|A(u(y,s))-A(v(x,t))\right|
\, \dif \mu(z).
\end{align*}
After another integration, this time w.r.t.~$\psi(x-y,t,s) \, \dif
w$, we then get that
\begin{align*}
& \iint_{Q_T^2} \sgn (u(y,s)-v(x,t)) \, \Big( \Levy^{\mu,r}[A(u(\cdot,s))](y)-\Levy^{\mu,r}[A(v(\cdot,t))](x) \Big) \, \psi \, \dif w \nonumber \\
& \leq \iint_{Q_T^2} \int_{|z|>r} \left|A(u(y+z,s))-A(v(x+z,t))\right| \, \psi(x-y,t,s) \, d\mu(z) \, \dif w\nonumber \\
& \quad -\iint_{Q_T^2} \int_{|z|>r} \left|A(u(y,s))-A(v(x,t))\right| \, \psi(x-y,t,s) \, d\mu(z) \, \dif w, \nonumber \\
& =: I+J.
\end{align*}
Note that these integrals are finite since 
$
\|A(u)\|_{C([0,T];L^1)}\leq
\|A'\|_{L^\infty}\|u\|_{C([0,T];L^1)}
$ ($A$ is Lipschitz-continuous
and $0$ at $0$) and by Fubini (note the convolution integrals in $x$
and $y$),
$$I,J\leq
\Big(\|A(u)\|_{C([0,T];L^1)}+\|A(v)\|_{C([0,T];L^1)}\Big)\|\psi\|_{L^1(\R^d
\times (0,T)^2 )}\int_{|z|>r}\dif \mu(z).$$
We then change variables $(z,x,t,y,s) \rightarrow
(z,x+z,t,y+z,s)$ in $I$,
\begin{align*}
I&=\iint_{Q_T} \int_{|z|>r}  \left|A(u(y,s))-A(v(x,t))\right| \, \psi(x-z-(y-z),t,s)\, \dif \mu(z) \, \dif w
,
\end{align*}
to find that $I+J=0$ and the proof  is complete.
\end{proof}

\begin{lemma}
\label{Kuz:lem} Under the assumptions of Lemma \ref{lem:kuznetsov},
\begin{align*}I=\iint_{Q_{T}^2}
|A(v(x,t))-A(u(y,s))|\,\Levy^{\mu^\ast}_{r}[\phi^{\epsilon,\delta}(x,t,\cdot,s)](y)
\, \dif w\leq C_{\epsilon } \, \int_{0<|z|\leq r} |z|^2 \, \dif
\mu(z),
\end{align*}
where~$C_\epsilon>0$ does not depend on~$r>0$.
\end{lemma}
\begin{proof}
Easy computations show that
\begin{align*}
& \Levy^{\mu^\ast}_{r}[\phi^{\epsilon,\delta}(x,t,\cdot,s)](y)  \\
& =\theta_\delta(t-s) \, \int_{0<|z|\leq r} \bar{\theta}_\epsilon(x-y-z)-\bar{\theta}_\epsilon(x-y)+z \cdot D  \bar{\theta}_\epsilon(x-y) \, \mathbf{1}_{|z| \leq 1} \, \dif \mu^\ast(z)  \\
& = \theta_\delta(t-s) \, \int_{0<|z|\leq r} \bar{\theta}_\epsilon(x-y+z)-\bar{\theta}_\epsilon(x-y)-z \cdot D  \bar{\theta}_\epsilon(x-y) \, \mathbf{1}_{|z| \leq 1} \, \dif \mu(z)  \\
& = \theta_\delta(t-s) \,
\Levy^{\mu}_{r}[\bar{\theta}_\epsilon](x-y),
\end{align*}

and by Fubini (there are again convolution integrals in~$I$!),
\begin{align*}
I&\leq \iint_{Q_T^2} |A(u(y,s) )-A(v(x,t) )| \, \theta_\delta(t-s)
\left|\Levy^{\mu}_{r}[\bar{\theta}_\epsilon](x-y) \right| \, \dif w
\\
&\leq
\Big(\|A(u)\|_{L^1(Q_T)}+\|A(v)\|_{L^1(Q_T)}\Big)\|\theta_\delta
\, \Levy^{\mu}_{r}[\bar{\theta}_\epsilon]\|_{L^1(\R^{d+1})}\\
&\leq
T\|A'\|_{L^\infty}\Big(\|u\|_{C\left([0,T];L^1\right)}+\|v\|_{C\left([0,T];L^1\right)}\Big)\|\theta_\delta
\, \Levy^{\mu}_{r}[\bar{\theta}_\epsilon]\|_{L^1(\R^{d+1})}.
\end{align*}
By classical properties of approximate units
and Lemma
\ref{nat:def-singular}, 
\begin{align*}
&\|\theta_\delta \,
\Levy^{\mu}_{r}[\bar{\theta}_\epsilon]\|_{L^1(\R^{d+1})} =
\underbrace{\| \theta_\delta \|_{L^1(\R)}}_{=1} \,
\|\Levy^{\mu}_{r}[\bar{\theta}_\epsilon]\|_{L^1(\R^d)}\\
&\leq
\|\bar \theta_\epsilon\|_{W^{2,1}(\R^d)} \int_{0<|z|\leq
    r} |z|^2 \, \dif \mu(z).
\end{align*}
The proof is complete.

\end{proof}

\section{Proofs of the main results}\label{nat:sec-proofs}

The proofs of this section use the so-called doubling of variables
technique of Kruzhkov~\cite{Kru70} along
with ideas from~\cite{Jakobsen/Karlsen,Kuznetsov}; for other relevant references, see also e.g.~\cite{Alibaud,Cifani/Jakobsen,KaUl10} for nonlocal equations. It consists in considering~$u$ as
a function of the new variables~$(y,s)$ and using the approximate
units~$\phi^{\epsilon,\delta}$ in~\eqref{nat:test-kuznetsov} as test
functions. For brevity, we do not specify anymore the variables
of~$u=u(y,s)$,~$v=v(x,t)$
and~$\phi^{\epsilon,\delta}=\phi^{\epsilon,\delta}(x,t,y,s)$ when
the context is clear; recall also that~$\dif x \, \dif
t \, \dif y \, \dif s$ is denoted by~$\dif w$.

\subsection{Proof of Lemma~\ref{lem:kuznetsov}}
 Let~$(x,t)
  \in Q_T$ be fixed and $u=u(y,s)$, $k=v(x,t)$,
  and~$\phi(y,s):=\phi^{\epsilon,\delta}(x,t,y,s)$.
   The entropy inequality for $u$ (see \eqref{entropy_ineq})
  then takes the form
\begin{align*}
&\int_{Q_{T}} |u-v|\,\partial_s \phi^{\epsilon,\delta}+\left(q_f(u,v)+|A(u)-A(v)| \, b_r^{\mu^\ast}\right)\cdot D_y \phi^{\epsilon,\delta} \, \dif y \, \dif s\\
&+\int_{Q_{T}}
|A(u)-A(v)|\,\Levy^{\mu^\ast}_{r}[\phi^{\epsilon,\delta}(x,t,\cdot,s)](y) \, \dif y \, \dif s\\
&+\int_{Q_{T}}\sgn (u-v)\,\Levy^{\mu,r}[A(u(\cdot,s))](y) \,\phi^{\epsilon,\delta} \, \dif y \, \dif s\\
&-\int_{\R^d} |u(y,T)-v(x,t)|\,\phi^{\epsilon,\delta}(x,t,y,T) \, \dif y\\
&+\int_{\R^d} |u_0(y)-v(x,t)|\,\phi^{\epsilon,\delta}(x,t,y,0) \,
\dif y \geq 0.
\end{align*}
We integrate this inequality w.r.t.~$(x,t) \in Q_T$, noting
that~$q_f$ in~\eqref{nat:def-flux} is symmetric, and that
$\partial_s \phi^{\epsilon,\delta}=-\partial_t
\phi^{\epsilon,\delta}$ and~$D_y \phi^{{\epsilon},\delta}=-D_x
\phi^{{\epsilon},\delta}$ by \eqref{nat:test-kuznetsov}.
Consequently we find that
\begin{equation}\label{nat:kuz-1}
\begin{split}
  & I_1+\dots+I_5\\
  & :=-\iint_{Q_{T}^2} |u-v|\,\partial_t \phi^{\epsilon,\delta}+\bigg(q_f(v,u)+ |A(u)-A(v)| \, b_r^{\mu^\ast} \bigg) \cdot D_x \phi^{{\epsilon},\delta} \, \dif w\\
  & \quad +\iint_{Q_{T}^2} |A(u)-A(v)|\,\Levy^{\mu^\ast}_{r}[\phi^{\epsilon,\delta}(x,t,\cdot,s)](y) \, \dif w \\
  & \quad +\iint_{Q_{T}^2} \sgn (u-v)\,\Levy^{\mu,r}[A(u(\cdot,s))](y) \,\phi^{\epsilon,\delta} \, \dif w\\
  & \quad -\iint_{Q_T \times \R^d} |u(y,T)-v(x,t)|\,\phi^{\epsilon,\delta}(x,t,y,T) \, \dif x \, \dif t \, \dif y\\
  & \quad +\iint_{Q_T \times \R^d}
  |u_0(y)-v(x,t)|\,\phi^{\epsilon,\delta}(x,t,y,0) \, \dif x \, \dif t
  \, \dif y \geq 0.
\end{split}
\end{equation}
Note that the terms in the inequality above are well-defined since
they are all essentially of the form of convolution integrals of
$L^1$-functions. See Lemma \ref{nat:def-singular}, Remark~\ref{rem:splitting-b},
and the discussions in the proofs of Lemmas
\ref{nat:pro-kato-d} and~\ref{Kuz:lem} for more details.

A classical computation from \cite{Kuznetsov} reveals that
\begin{equation*}
\begin{split}
&I_4+I_5
-\iint_{\R^d \times Q_T} |u(y,s)-v(x,T)|\,\phi^{\epsilon,\delta}(x,T,y,s) \, \dif x \, \dif y \, \dif s\\
&\quad+\iint_{\R^d \times Q_T} |u(y,s)-v_0(x)|\,\phi^{\epsilon,\delta}(x,0,y,s) \, \dif x \, \dif y \, \dif s\\
&\leq -\|u(T)-v(T)\|_{L^1(\R^d)}+\|u_0-v_0\|_{L^1(\R^d)}\\
&\quad+ \epsilon \, C_{\tilde\theta} \, |u_0|_{BV(\R^d)}+2\omega_{u}
(\delta) \vee \omega_{v} (\delta),
\end{split}
\end{equation*}
where $C_{\tilde\theta}$ is as in Lemma \ref{lem:kuznetsov}. 
Lemma \ref{lem:kuznetsov} now follows from
\eqref{nat:kuz-1} and the above estimates on $I_4$ and $I_5$.


\subsection{Proof of Theorem~\ref{th:nonlin}}

The proof uses the Kuznetsov lemma, and morally speaking it amounts to
subtracting the $u(y,s)$ and $v(x,t)$ equations,
multiplying by $\sgn(u-v)$, 
and then applying both new and classical tricks to arrive at an
$L^1$-estimate of $|u-v|$. We expect to see terms involving
$$\sgn(u-v)\Big(\Levy^{\mu,r}[A(u)]-\Levy^{\mu,r}[B(v)]\Big),$$
and naively we can write this as 
$$\sgn(u-v)\Levy^{\mu,r}[(A-B)(u)]+\sgn(u-v)\Levy^{\mu,r}[B(u)-B(v)].$$
These terms are estimated by Kato type inequalities (see Lemmas
\ref{nat:ipp} and \ref{nat:pro-kato-d}), the first term 
should give the dependence on $A-B$ while the second term is a
nonpositive term that also appears in the uniqueness
proof. {\em The problem with this approach is that we can not apply Kato
for the first term because $A-B$ then have to be monotone!} 

There are different ways to overcome this monotonicity problem, and we
have chosen to adapt ideas from \cite{Jakobsen/Karlsen} -- a
paper on continuous dependence estimates for fully
nonlinear Bellman-Isaacs type of equations via viscosity solution techniques. 
We consider the region where~$A' \geq B'$ and
its complementary.  Let $E_\pm$ be sets satisfying:
\begin{equation}\label{nat:Epm-nl}
\begin{cases}
\mbox{$E_\pm \subseteq \R$ are Borel sets};\\
\cup_\pm E^\pm=\R \mbox{ and } \cap_\pm E_\pm=\emptyset;\\
 \R\setminus\mathrm{supp}(A'-B')^\mp \subseteq E_\pm.
\end{cases}
\end{equation}
For all~$u \in \R$, we define
\begin{equation}\label{nat:cnl-not}
\begin{split}
& A_\pm(u):=\int_{0}^u A'(\tau) \, \mathbf{1}_{E_\pm}(\tau) \, \dif \tau,\\
& B_\pm(u):=\int_{0}^u B'(\tau) \, \mathbf{1}_{E_\pm}(\tau) \, \dif \tau,\\
& C_\pm(u):=\pm (A_\pm(u)-B_\pm(u)).
\end{split}
\end{equation}
These functions satisfy the following properties:
\begin{lemma}\label{nat:lem-dnl}
Under the assumptions of Theorem~\ref{th:nonlin},
\begin{itemize}
\item[(i)]$A=A_++A_-$ and~$B=B_++B_-$;
\item[(ii)] $A_{\pm},B_{\pm},C_\pm$ satisfy~\eqref{Aflux}, in
  particular, they are monotone;
\item[(iii)] $\sum_\pm |C_\pm(u)|_{L^1(0,T;BV)} \leq  \|A'-B'\|_{L^\infty(\R)} \, |u|_{L^1(0,T;BV)}$;
\item[(iv)] for all~$z \in \R^d \setminus \{0\}$,
$$
\sum_\pm \|C_\pm(u(\cdot+z,\cdot))-C_\pm(u)\|_{L^1(Q_T)} \leq
\|A'-B'\|_{L^\infty(\R)} \, \|u(\cdot+z,\cdot)-u\|_{L^1(Q_T)}.
$$
\end{itemize}
\end{lemma}
The proofs of~(i) and~(ii) are immediate, whereas~(iii) and~(iv)
follow from standard arguments for Lipschitz-continuous
and~$BV$-functions (see e.g.~\cite{Perthame,EvGa92,Rud74}); the
details are left to the reader. 

In the proof below, $A^\pm-B^\pm$ will be the monotone functions
replacing the nonmonotone function $A-B$ of the formal argument above.

\begin{proof}[Proof of Theorem~\ref{th:nonlin}]
Let us divide the proof into several steps.

\medskip

\noindent {\bf 1.}  We argue as in the beginning of the proof of
Lemma~\ref{lem:kuznetsov} changing the roles of~$u$ and~$v$. We fix
$(y,s)$ and take~$k=u(y,s)$ and
$\phi^{\epsilon,\delta}=\phi^{\epsilon,\delta}(x,t,y,s)$ in the
entropy inequality for $v=v(x,t)$ to find that
\begin{equation*}
\begin{split}
& \iint_{Q_{T}^2} |v-u|\,\partial_t \phi^{\epsilon,\delta}+\left(q_g(v,u)+ |B(v)-B(u)| \, b_r^{\mu^\ast} \right) \cdot D_x \phi^{\epsilon,\delta} \, \dif w\\
&+\iint_{Q_{T}^2} |B(v)-B(u)|\,\Levy^{\mu^\ast}_{r}[\phi^{\epsilon,\delta}(\cdot,t,y,s)](x) \, \dif w \\
&+\iint_{Q_{T}^2} \sgn (v-u)\,\Levy^{\mu,r}[B(v(\cdot,t))](x) \,\phi^{\epsilon,\delta} \, \dif w\\
& -\iint_{\R^d \times Q_T} |v(x,T)-u(y,s)|\,\phi^{\epsilon,\delta}(x,T,y,s) \, \dif x \, \, \dif y \, \dif s\\
& +\iint_{\R^d \times Q_T}
|v_0(x)-u(y,s)|\,\phi^{\epsilon,\delta}(x,0,y,s) \, \dif x \, \,
\dif y \, \dif s \geq 0.
\end{split}
\end{equation*}
Then we add this inequality and inequality \eqref{kuz} in
Lemma~\ref{lem:kuznetsov},
\begin{equation}\label{nat:cnl-1}
\begin{split}
&\|u(T)-v(T)\|_{L^{1}(\mathbb{R}^d)}\\
& \leq \|u_0-v_0\|_{L^{1}(\mathbb{R}^d)}+\epsilon \, C_{\tilde\theta} \, |u_0|_{BV(\R^d)}+2\,\omega_u(\delta) \vee \omega_v(\delta)\\
& \quad +\underbrace{\iint_{Q_{T}^2}(q_g-q_f)(v,u) \cdot D_x  \phi^{\epsilon,\delta} \, \dif w}_{=:I_1}\\
&\quad +\underbrace{\iint_{Q_{T}^2}|B(v)-B(u)| \,\mathcal{L}^{\mu^\ast}_r[\phi^{\epsilon,\delta}(\cdot,y,t,s)](x)\, \dif w}_{=:I_2}\\
&\quad+\underbrace{\iint_{Q_{T}^2}
|A(v)-A(u)|\,\Levy^{\mu^\ast}_{r}[\phi^{\epsilon,\delta}(x,t,\cdot,s)](y)
\, \dif w}_{=:I_2'} \\
&\quad
+\underbrace{\iint_{Q_{T}^2} \Big(|B(v)-B(u)|-|A(v)-A(u)|\Big) \, b_r^{\mu^\ast} \, \cdot D_x \phi^{\epsilon,\delta} \, \dif w}_{=:I_3}\\
&\quad +\underbrace{\iint_{Q_{T}^2}\sgn(v-u)\,
  \Big(\mathcal{L}^{\mu,r}[B(v(\cdot,t))](x)-\mathcal{L}^{\mu,r}[A(u(\cdot,s))](y)\Big)
  \,\phi^{\epsilon,\delta}\, \dif w}_{=: I_4},
\end{split}
\end{equation}
where $r,\epsilon>0$, $0<\delta<T$, and $C_{\tilde\theta}>0$ only
depends on the kernel $\tilde \theta_d $ from
\eqref{nat:smooth-kernel}.
\medskip

\noindent {\bf 2.} It is standard to estimate $I_1$ (cf.
e.g.~\cite{Da:Book,Kuznetsov}), and $I_2+I_2'$ can be estimated by
Lemma~\ref{Kuz:lem},
\begin{align}\label{nat:end}
I_1 & \leq  |u_0|_{BV(\R^d)} \, T \, \|f'-g'\|_{L^\infty(\R,\R^d)},\\
I_2+I_2' & \leq   C_\epsilon \, \int_{0<|z| \leq r} |z|^2 \, \dif
\mu(z),\label{nat:end2}
\end{align}
where~$C_\epsilon$ does not depend on~$r>0$. 
Now we focus on~$I_3$
and~$I_4$.

\medskip

\noindent\textbf{3.} {\em Cutting w.r.t.~$E_\pm$.} We split $I_3$
and~$I_4$ into four new terms using the sets $E_\pm$, see
\eqref{nat:Epm-nl}--\eqref{nat:cnl-not}. By
Lemma~\ref{nat:lem-dnl}~(i), $I_4$ can be written as
\begin{align*}
I_4=\sum_\pm \iint_{Q_{T}^2}\sgn(v-u)\,
  \Big(\mathcal{L}^{\mu,r}[B_\pm(v(\cdot,t))](x)-\mathcal{L}^{\mu,r}[A_\pm(u(\cdot,s))](y)\Big)
  \,\phi^{\epsilon,\delta}\, \dif w.
\end{align*}
By Lemma~\ref{nat:lem-dnl}~(ii), we can apply twice
Lemma~\ref{nat:pro-kato-d} with $B_+$ and $A_-$ instead of $A$,
followed by the definitions of $C_\pm$, see \eqref{nat:cnl-not}, to
show that
\begin{align}
 I_4\leq & \,\iint_{Q_{T}^2}\sgn(v-u)\,
\mathcal{L}^{\mu,r}\Big[B_+(u(\cdot,s))-A_+(u(\cdot,s))\Big](y)
\,\phi^{\epsilon,\delta}\, \dif w\nonumber\\
&+\iint_{Q_{T}^2}\sgn(v-u)\,
\mathcal{L}^{\mu,r}\Big[B_-(v(\cdot,t))-A_-(v(\cdot,t))\Big](x)
\,\phi^{\epsilon,\delta}\, \dif w\nonumber\\
= &\,\iint_{Q_{T}^2}\sgn(u-v)\,
\mathcal{L}^{\mu,r}\big[C_+(u(\cdot,s))\big](y)
\,\phi^{\epsilon,\delta}\, \dif w\nonumber\\
&+\iint_{Q_{T}^2}\sgn(v-u)\,
\mathcal{L}^{\mu,r}\big[C_-(v(\cdot,t))\big](x)
\,\phi^{\epsilon,\delta}\, \dif w\nonumber\\
=:& \ I_4^++I_4^-.\label{nat:cnl-2-bis}
\end{align}
Note that it is crucial to have $u$ in the first term and $v$ in the
second -- otherwise we will not be able to apply the Kato inequality
later on!

We now consider $I_3$. By~\eqref{nat:key},
Lemma~\ref{nat:lem-dnl}~(i)--(ii), the formula~$D_x
\phi^{\epsilon,\delta}=-D_y \phi^{\epsilon,\delta}$, and the
definitions $D_+=D_y$ and $D_-=D_x$, it follows that
\begin{align*}
& \Big(|B(v)-B(u)|-|A(v)-A(u)|\Big) D_x \phi^{\epsilon,\delta}\\
& = \sgn (u-v) \, \Big\{\left(A(u)-B(u)\right)-\left(A(v)-B(v)\right) \Big\} \, D_y \phi^{\epsilon,\delta}\\
& = \sum_{\pm} \sgn (u-v) \, \Big\{\pm\left(A_\pm(u)-B_\pm(u)\right)\mp\left(A_\pm(v)-B_\pm(v)\right) \Big\} \, D_\pm \phi^{\epsilon,\delta}\\
& = \sum_{\pm} |C_\pm(u)-C_\pm(v)| \, D_\pm \phi^{\epsilon,\delta}.
\end{align*}
 We can then rewrite~$I_3$ as
\begin{equation}\label{nat:cnl-1-bis}
I_3= \sum_\pm \underbrace{\iint_{Q_T} |C_\pm(u)-C_\pm(v)| \,
b^{\mu^\ast}_r \cdot D_\pm \phi^{\epsilon,\delta} \, \dif
w}_{=:I_3^\pm}.
\end{equation}

\medskip

\noindent {\bf 4.} {\em Cutting w.r.t.~$z$.} We decompose $
\Levy^{\mu,r}$ into two new terms using a new cutting parameter
$r_1>r$. Let $\mu=\mu_1+\mu_{|_{|z|>r_1}}$ for
$$
\mu_1: =\mu_{|_{0<|z| \leq r_1}},
$$
and note that by \eqref{nat:def-decomposition-outer},
$\Levy^{\mu,r}=\Levy^{\mu_{1},r}+\Levy^{\mu,r_1}$. Then
\begin{equation}\label{nat:cnl-kato-bis}
\begin{split}
&I_4^+ = \underbrace{\iint_{Q_{T}^2}\sgn(u-v)\,
\mathcal{L}^{\mu_1,r}[C_+(u(\cdot,s))](y) \,\phi^{\epsilon,\delta}\,
\dif
w}_{=:I_5^+}\\
&\qquad\qquad+\iint_{Q_{T}^2}\sgn(u-v)\,
\mathcal{L}^{\mu,r_1}[C_+(u(\cdot,s))](y) \,\phi^{\epsilon,\delta}\,
\dif w.
\end{split}
\end{equation}
Since $C_+$ satisfies \eqref{Aflux} by Lemma~\ref{nat:lem-dnl}~(ii)
and $\mu_1$ clearly satisfies \eqref{levy-measure}, we can apply the
Kato type inequality in Lemma~\ref{nat:ipp} (with~$k=v(x,t)$ and
$A=C_+$) to show that
\begin{align*}
& I_5^+ =\int_{Q_{T}} \int_{Q_T} \sgn(u(y,s)-v(x,t))\, \mathcal{L}^{\mu_1,r}[C_+(u(\cdot,s))](y) \,\phi^{\epsilon,\delta}\, \dif y \, \dif s \, \dif x \, \dif t\\
& \leq \int_{Q_{T}} \int_{Q_T} |C_+(u(y,s))-C_+(v(x,t))|
\,\Levy^{\mu_1^\ast,r}[\phi^{\epsilon,\delta}(x,t,\cdot,s)](y)\,
\dif y \, \dif s \, \dif x \, \dif t.
\end{align*}
Adding~$I_3^+$ in the form \eqref{nat:cnl-1-bis} then gives
\begin{equation}\label{nat:cnl-4}
I_3^++I_5^+\leq \iint_{Q_{T}^2} |C_+(u)-C_+(v)| \left(
b_r^{\mu^\ast} \, \cdot D_y \phi^{\epsilon,\delta}
+\Levy^{\mu_1^\ast,r}[\phi^{\epsilon,\delta}(x,t,\cdot,s)](y)
\right) \dif w.
\end{equation}

Now easy computations show that
$$D_y \phi^{\epsilon,\delta}=-\theta_\delta(t-s) \, D
\bar{\theta}_\epsilon(x-y), \quad
\Levy^{\mu_1^{\ast},r}[\phi^{\epsilon,\delta}(x,t,\cdot,s)](y)=\theta_\delta(t-s)
\, \Levy^{\mu_1,r}[\bar{\theta}_{\epsilon}](x-y).$$ Hence by adding
and subtracting~$z \cdot D \bar{\theta}_{\epsilon} (x-y)$, we get
that
\begin{equation}
\begin{split}\label{nat:cnl-6}
& b_r^{\mu^\ast} \, \cdot D_y \phi^{\epsilon,\delta}+\Levy^{\mu_1^{\ast},r}[\phi^{\epsilon,\delta}(x,t,\cdot,s)](y)\\
& = \theta_\delta(t-s) \int_{r<|z|\leq r_1} \bar{\theta}_\epsilon(x-y+z)-\bar{\theta}_\epsilon(x-y)-z \cdot D \bar{\theta}_{\epsilon} (x-y) \, \dif \mu(z) \\
& \quad +\theta_\delta(t-s) \, D \bar{\theta}_{\epsilon} (x-y) \cdot
\underbrace{\left(-b_r^{\mu^\ast}+\int_{r<|z|\leq r_1} z \, \dif
\mu(z)\right)}_{=\sgn(r_1-1) \int_{r_1 \wedge (1 \vee r) <|z|\leq
r_1 \vee 1} z \, \dif \mu(z)},
\end{split}
\end{equation}
where the last equality comes from~\eqref{nat:def-decomposition-b}
and the change of variable~$z \rightarrow -z$. We
insert~\eqref{nat:cnl-6} into \eqref{nat:cnl-4} and combine the
resulting inequality with~\eqref{nat:cnl-kato-bis},
\begin{equation}\label{nat:cnl-8}
\begin{split}
& I_3^++I_4^+\leq \\
& \iint_{Q_{T}^2} |C_+(u)-C_+(v)|\\
& \quad\cdot\theta_\delta (t-s) \int_{r<|z|\leq r_1}  \bar{\theta}_\epsilon(x-y+z)-\bar{\theta}_\epsilon(x-y)-z \cdot D \bar{\theta}_{\epsilon} (x-y) \, \dif \mu(z) \, \dif w\\
&  +\iint_{Q_{T}^2} |C_+(u)-C_+(v)|\\
& \qquad\cdot \theta_\delta (t-s) \, D \bar{\theta}_\epsilon(x-y) \cdot \sgn(r_1-1) \int_{r_1 \wedge (1 \vee r) <|z|\leq r_1 \vee 1} z \, \dif \mu(z) \, \dif w\\
&  +\iint_{Q_{T}^2}\sgn(u-v)\, \mathcal{L}^{\mu,r_1}[C_+(u(\cdot,s))](y) \,\phi^{\epsilon,\delta}\, \dif w\\
& =:J_1^++J_2^++J_3^+.
\end{split}
\end{equation}
Similar arguments show that we can bound $I_3^-+I_4^-$ (see
\eqref{nat:cnl-2-bis}--\eqref{nat:cnl-1-bis}) as follows,
\begin{equation}\label{nat:cnl-8-bis}
\begin{split}
& I_3^-+I_4^-\leq  \\
& \iint_{Q_{T}^2} |C_-(v)-C_-(u)|\\
&\quad\cdot \theta_\delta (t-s) \int_{r<|z|\leq r_1} \bar{\theta}_\epsilon(x-y-z)-\bar{\theta}_\epsilon(x-y)+z \cdot D \bar{\theta}_{\epsilon} (x-y) \, \dif \mu(z) \, \dif w\\
& -\iint_{Q_{T}^2} |C_-(v)-C_-(u)| \\
&\qquad\cdot \theta_\delta (t-s) \, D \bar{\theta}_\epsilon(x-y) \cdot \sgn(r_1-1) \int_{r_1 \wedge (1 \vee r) <|z|\leq r_1 \vee 1} z \, \dif \mu(z) \, \dif w\\
& +\underbrace{\iint_{Q_{T}^2} \sgn(v-u)\,
  \mathcal{L}^{\mu,r_1}[C_-(v(\cdot,t))](x) \,\phi^{\epsilon,\delta}\,
  \dif w}_{\leq \iint_{Q_{T}^2} \sgn(v-u)\,
  \mathcal{L}^{\mu,r_1}[C_-(u(\cdot,s))](y) \,\phi^{\epsilon,\delta}\,
  \dif w\ \text{by Lemma \ref{nat:pro-kato-d}
  }}\\
& =:J_1^-+J_2^-+J_3^-.
\end{split}
\end{equation}

\medskip
\noindent {\bf 5.} {\em $L^1 \cap BV$-regularity.} It remains to
estimate~$J_i^\pm$ for $i=1,\dots,3$
in~\eqref{nat:cnl-8}--\eqref{nat:cnl-8-bis}. For~$J_1^\pm$
and~$J_2^\pm$, we integrate by parts to take
advantage of the~$BV$-regularity of~$u$. After
some technical computations detailed in Appendix~\ref{nat:app-tech}, we find that
\begin{align}
|J_1^\pm|&\leq \frac{1}{2\epsilon} \, \int_{\R^d}|D\tilde\theta_d|\dif x \, \int_{r<|z| \leq r_1} |z|^2 \, \dif \mu(z) \, |C_\pm(u)|_{L^1(0,T;BV)},\label{nat:esti-sing-1}\\
|J_2^\pm|&\leq \left|\int_{r_1 \wedge (1 \vee r) <|z|\leq r_1 \vee
1} z \, \dif \mu(z) \right| \,  |C_\pm(u)|_{L^1(0,T;BV)},\label{nat:esti-sing-2}
\end{align}
and hence
\begin{align*}
\sum_\pm (J_1^\pm+J_2^\pm) \leq \,&\frac{1}{2\epsilon} \,
\int_{\R^d}|D\tilde\theta_d|\dif x  \int_{r<|z| \leq r_1} |z|^2 \, \dif \mu(z) \, \sum_\pm |C_\pm(u)|_{L^1(0,T;BV)}\\
&+ \left|\int_{r_1 \wedge (1 \vee r) <|z|\leq r_1 \vee 1} z \, \dif
\mu(z) \right| \sum_\pm |C_\pm(u)|_{L^1(0,T;BV)}.
\end{align*}
By Lemma~\ref{nat:lem-dnl}~(iii) and a priori estimates for $u$, cf.
\eqref{nat:nonincrease}, we see that
\begin{align}
\sum_\pm (J_1^\pm+J_2^\pm) \leq  & \,\frac{1}{2\epsilon} \,  \int_{\R^d}|D\tilde\theta_d|\dif x  \, \underbrace{|u|_{L^1(0,T;BV)}}_{\leq |u_0|_{BV(\R^d)} \, T} \int_{r<|z| \leq r_1} |z|^2 \, \dif \mu(z) \, \|A'-B'\|_{L^\infty(\R)}\nonumber\\
&+ \underbrace{|u|_{L^1(0,T;BV)}}_{\leq |u_0|_{BV(\R^d)} \, T}
\left|\int_{r_1 \wedge (1 \vee r) <|z|\leq r_1 \vee 1} z \, \dif
\mu(z) \right| \, \|A'-B'\|_{L^\infty(\R)}.\label{nat:cnl-9-bis}
\end{align}

Let us now estimate~$J_3^+$ in~\eqref{nat:cnl-8}. Easy computations
(see the proofs of Lemmas~\ref{nat:pro-kato-d}--\ref{Kuz:lem})
show that
$$
J_{3}^+ \leq \|\theta_\delta \,
\bar{\theta}_\epsilon\|_{L^1(\R^{d+1})} \,
\|\mathcal{L}^{\mu,r_1}[C_+(u)]\|_{L^1(Q_T)}.
$$
Let us recall that~$\|\theta_\delta \,
\bar{\theta}_\epsilon\|_{L^1(\R^{d+1})}=\|\theta_\delta\|_{L^1(\R)}
\, \|\bar{\theta}_\epsilon\|_{L^1(\R^{d})}=1$, and then 
\begin{align*}
J_{3}^+ \leq 
\, \int_{0}^T 
\int_{|z|>r_1}
\|C_+(u(\cdot+z,s))-C_+(u(\cdot,s))\|_{L^1(\R^d)} \, \dif \mu(z)
 \dif s.
\end{align*}
Since $C_+(u) \in L^\infty \cap C([0,T];L^1) 
$, $(z,s) \rightarrow
\|C_+(u(\cdot+z,s))-C_+(u(\cdot,s))\|_{L^1(\R^d)}$ 
is a continuous function, hence Borel and~$\dif \mu(z) \, \dif s$-measurable. Thus, we may change
the order of the integration to find
\begin{align*}
J_{3}^+ \leq 
\int_{|z|>r} \|C_+(u(\cdot+z,\cdot))-C_+(u)\|_{L^1(Q_T)} \, \dif
\mu(z).
\end{align*}
We get a similar estimates for $J_3^-$ and find by
Lemma~\ref{nat:lem-dnl}~(iii)--(iv) and~\eqref{nat:nonincrease} that
\begin{align}
 \sum_\pm J_{3}^\pm
& \leq 
\int_{|z|>r_1} \sum_\pm \|C_\pm(u(\cdot+z,\cdot))-C_\pm(u)\|_{L^1(Q_T)} \, \dif \mu(z), \nonumber \\
& \leq 
\int_{|z|>r_1}
\underbrace{\|u(\cdot+z,\cdot))-u\|_{L^1(Q_T)}}_{\leq T \,
\|u_0(\cdot+z)-u_0\|_{L^1(\R^d)}} \, \dif \mu(z) \,
\|A'-B'\|_{L^\infty(\R^d)}.\label{nat:cnl-10-bis}
\end{align}
The last inequality (under the bracket) comes
from~\eqref{nat:L1-contraction} applied to the
solution~$u(\cdot+z,\cdot)$ of~\eqref{1} with initial data
$u_0(\cdot+z)$.

\medskip

\noindent {\bf 6.} {\em  Conclusion.} By
\eqref{nat:cnl-2-bis}--\eqref{nat:cnl-1-bis}
and~\eqref{nat:cnl-8}--\eqref{nat:cnl-8-bis}, $I_3+I_4 \leq \sum_\pm
\sum_{i=1}^3 J_i^\pm$. Therefore we may estimate \eqref{nat:cnl-1}
by \eqref{nat:end}--\eqref{nat:end2} and
\eqref{nat:cnl-9-bis}--\eqref{nat:cnl-10-bis}.
 For all~$r_1>r>0$, $\epsilon>0$, and $T>\delta>0$, we find that
\begin{equation}
\begin{split}\label{nat:cnl-13}
& \|u(T)-v(T)\|_{L^{1}(\mathbb{R}^d)}\\
& \leq \|u_0-v_0\|_{L^{1}(\mathbb{R}^d)} + |u_0|_{BV(\R^d)} \, T \, \|f'-g'\|_{L^\infty(\R,\R^d)}\\
& \quad + \epsilon \, C_{\tilde\theta} \, |u_0|_{BV(\R^d)}+ 2\,\omega_u(\delta) \vee \omega_v(\delta)+C_\epsilon \, \int_{0<|z| \leq r} |z|^2 \, \dif \mu(z)\\
& \quad +\frac{1}{2\epsilon} \, \int_{\R^d}|D\tilde\theta_d|\dif x \, |u_0|_{BV(\R^d)} \, T  \int_{r<|z| \leq r_1} |z|^2 \, \dif \mu(z) \, \|A'-B'\|_{L^\infty(\R^d)}\\
& \quad +|u_0|_{BV(\R^d)} \, T \left|\int_{r_1 \wedge (1 \vee r) <|z|\leq r_1 \vee 1} z \, \dif \mu(z) \right| \, \|A'-B'\|_{L^\infty(\R^d)}\\
& \quad +T \int_{|z|>r_1} \|u_0(\cdot+z)-u_0\|_{L^1(\R^d)} \, \dif
\mu(z) \, \, \|A'-B'\|_{L^\infty(\R^d)},
\end{split}
\end{equation}
where~$C_\epsilon>0$ does not depend on~$r>0 $.

To finish, we first pass to the limit as~$r \rightarrow 0$
in~\eqref{nat:cnl-13}. By the dominated convergence theorem, the
result is equivalent to setting $r=0$ in each term, and in
particular the term $C_\epsilon \, \int_{0<|z| \leq r} |z|^2 \, \dif
\mu(z)$ vanishes. Secondly, we pass to the limit as~$\delta
\rightarrow 0$ to get rid of the term $2\,\omega_u(\delta) \vee
\omega_v(\delta)$. Finally, we optimize the remaining terms
w.r.t.~$\epsilon>0$ by using the formula~$ \min_{\epsilon>0}
\left(\epsilon \, a+\frac{b}{\epsilon} \right)=2\sqrt{ab} $
(for~$a,b \geq 0$). This gives us the following continuous
dependence estimate: For all~$r_1>0$,
\begin{equation}\label{nat:cnl-14}
\begin{split}
& \|u-v\|_{C\left([0,T];L^{1}\right)} \leq\|u_0-v_0\|_{L^{1}(\mathbb{R}^d)}+|u_0|_{BV(\R^d)} \, T \, \|g'-f'\|_{L^\infty(\R,\R^d)}\\
& \quad + 2\sqrt{\frac{1}{2} \, C_{\tilde\theta} \,
\int_{\R^d}|D\tilde\theta_d|\dif x \, |u_0|^2_{BV(\R^d)} \, T
\int_{0<|z| \leq r_1} |z|^2 \, \dif \mu(z)
 \, \|A'-B'\|_{L^\infty(\R)}}\\
& \quad + |u_0|_{BV(\R^d)} \, T\, \left|\int_{r_1 \wedge 1 < |z|
\leq r_1 \vee 1} z \, \dif \mu(z)\right|
 \,
\|A'-B'\|_{L^\infty(\R)}\\
& \quad +  T \, \int_{|z| \geq r_1} \|u_0(\cdot+z)-u_0
\|_{L^1(\R^d)} \, \dif \mu(z) \, \|A'-B'\|_{L^\infty(\R)},
\end{split}
\end{equation}
where $\tilde\theta_d$ is an arbitrary approximate unit
\eqref{nat:smooth-kernel} and $C_{\tilde\theta}=2 \int_{\R^d} |x| \,
\tilde\theta_d(x) \, \dif x$ by Lemma \ref{lem:kuznetsov}.

Let $\tilde\theta_d=\theta_n$ where $\{\theta_n\}_{n \in \N}$ is a
sequence of kernels s.t.~$\theta_n$
satisfies~\eqref{nat:smooth-kernel}, $\theta_n \rightarrow
{\omega_d}^{-1} \mathbf{1}_{|\cdot |<1}$ in~$L^1$, and $\int_{\R^d}
|D\theta_n| \, \dif x \rightarrow {\omega_d}^{-1}
|\mathbf{1}_{|\cdot|<1}|_{BV(\R^d)} $. Here $\omega_d$ is the volume
of the unit ball in $\R^d$. Note that the~$BV$-semi-norm of the
indicator function of the unit ball is equal to the surface area
of the unit sphere,~i.e. $|\mathbf{1}_{|\cdot |<1}|_{BV(\R^d)}
=d\omega_d$. Moreover, we have
$$\int_{\R^d}|x||\theta_n(x)| \, \dif
x\ra \frac1{\omega_d}\int_{|x|<1}|x| \, \dif x=\frac{d}{d+1}.$$
 The proof of~\eqref{nat:first-app} is then complete after passing to the limit as $n \rightarrow +\infty$
in \eqref{nat:cnl-14}.
\end{proof}

\subsection{Proof of Theorem~\ref{nat:th-levy}}

We argue step by step as in the proof of Theorem~\ref{th:nonlin}.
This time,~$E_\pm$ are taken such as
\begin{equation}\label{nat:Epm-m}
\begin{cases}
\mbox{$E_\pm \subseteq \R^d \setminus \{0\}$ are Borel sets};\\
\cup_\pm E^\pm=\R^d \setminus \{0\} \mbox{ and } \cap_\pm E_\pm=\emptyset;\\
\left(\R^d \setminus \{0\} \right)\setminus\mathrm{supp}(\mu-\nu)^\mp \subseteq E_\pm.
\end{cases}
\end{equation}
Let $\mu_\pm$ and $\nu_\pm$ denote the restrictions of $\mu$ and
$\nu$ to $E_\pm$. It is clear that
\begin{equation}\label{nat:cut-levy}
\begin{cases}
\mu=\sum_\pm \mu_\pm\quad \text{and}\quad  \nu=\sum_\pm \nu_\pm,\\
\pm (\mu_\pm-\nu_\pm)=(\mu-\nu)^\pm,\\
\mbox{$\mu_\pm,\nu_\pm,\ \text{and}\ \pm (\mu_\pm-\nu_\pm)$ all
satisfy~\eqref{levy-measure}.}
\end{cases}
\end{equation}

\begin{proof}[Proof of Theorem~\ref{nat:th-levy}]\
\smallskip

\noindent {\bf 1.}  We apply Lemma \ref{lem:kuznetsov} with $A=B$,
but different L\'evy measures~$\mu$ and~$\nu$, along with the
entropy inequality for $v$ to show that for all~$r,\epsilon>0$,
$0<\delta<T$
\begin{equation}\label{nat:cnl-1-levy}
\begin{split}
&\|u(T)-v(T)\|_{L^{1}(\mathbb{R}^d)}\\
& \leq \|u_0-v_0\|_{L^{1}(\mathbb{R}^d)}+\epsilon \, C_{\tilde\theta} \, |u_0|_{BV(\R^d)}+2\,\omega_u(\delta) \vee \omega_v(\delta)\\
& \quad +\iint_{Q_{T}^2}(q_g-q_f)(v,u) \cdot D_x  \phi^{\epsilon,\delta} \, \dif w\\
&\quad +\iint_{Q_{T}^2}|A(v)-A(u)| \,\mathcal{L}^{\nu^\ast}_r[\phi^{\epsilon,\delta}(\cdot,y,t,s)](x)\, \dif w\\
&\quad+\iint_{Q_{T}^2}
|A(v)-A(u)|\,\Levy^{\mu^\ast}_{r}[\phi^{\epsilon,\delta}(x,t,\cdot,s)](y)
\, \dif w\\
&\quad
+\underbrace{\iint_{Q_{T}^2} |A(v)-A(u)| \, \left(b_r^{\nu^\ast}-b_r^{\mu^\ast}\right) \, \cdot D_x \phi^{\epsilon,\delta} \, \dif w}_{=:I_3}\\
&\quad +\underbrace{\iint_{Q_{T}^2}\sgn(v-u)\,
  \Big(\mathcal{L}^{\nu,r}[A(v(\cdot,t))](x)-\mathcal{L}^{\mu,r}[A(u(\cdot,s))](y)\Big)
  \,\phi^{\epsilon,\delta}\, \dif w}_{=: I_4},
\end{split}
\end{equation}
where~$C_\epsilon>0$ does not depend on~$r>0 $. Except for $I_3$ and
$I_4$, the other terms were estimated in the proof of Theorem
\ref{th:nonlin}.
\medskip

\noindent {\bf 2.} {\em Cutting w.r.t.~$E_\pm$.} We use the notation
introduced in \eqref{nat:Epm-m}. We apply Lemma~\ref{nat:pro-kato-d}
twice with $\nu_+$ and $\mu_-$ instead of $\mu$, along with
linearity of  $\Levy^{\mu,r}$ in $\mu$, see
\eqref{nat:def-decomposition-inner}, to see that
\begin{align}
I_4&=\sum_\pm\iint_{Q_{T}^2}\sgn(v-u)\,
  \Big(\mathcal{L}^{\nu_\pm,r}[A(v(\cdot,t))](x)-\mathcal{L}^{\mu_\pm,r}[A(u(\cdot,s))](y)\Big)
  \,\phi^{\epsilon,\delta}\, \dif w\nonumber\\
&\leq \iint_{Q_{T}^2}\sgn(v-u)\,\Big(
\mathcal{L}^{\nu_+,r}[A(u(\cdot,s))](y)-\mathcal{L}^{\mu_+,r}[A(u(\cdot,s))](y)\Big)
\,\phi^{\epsilon,\delta}\, \dif w\nonumber\\
&\quad + \iint_{Q_{T}^2}\sgn(v-u)\,\Big(
\mathcal{L}^{\nu_-,r}[A(v(\cdot,t))](x)-\mathcal{L}^{\mu_-,r}[A(v(\cdot,t))](x)\Big)
\,\phi^{\epsilon,\delta}\, \dif w\nonumber\\
&=\iint_{Q_{T}^2}\sgn(u-v)\,\mathcal{L}^{\mu_+-\nu_+,r}[A(u(\cdot,s))](y)
\,\phi^{\epsilon,\delta}\, \dif w\nonumber\\
&\quad+
\iint_{Q_{T}^2}\sgn(v-u)\,\mathcal{L}^{-(\mu_--\nu_-),r}[A(v(\cdot,t))](x)
\,\phi^{\epsilon,\delta}\, \dif w\nonumber\\
&=:I_4^++I_4^-.\label{II4}
\end{align}
Again, it is crucial to have $u$ in $I_4^+$ and $v$ in $I_4^-$ in
order to use Kato's inequality later on.

Let us now consider $I_3$. By \eqref{nat:def-decomposition-b} and
\eqref{nat:measure-dual}, $b_r^{\mu}$ and $\mu^\ast$ are linear
w.r.t $\mu$. Easy computations using \eqref{nat:cut-levy} then leads
to
$$
\left(b_r^{\nu^\ast}-b_r^{\mu^\ast} \right) \cdot D_x
\phi^{\epsilon,\delta}= \sum_\pm b_r^{\pm (\mu_\pm-\nu_\pm)^\ast}
\cdot D_\pm \phi^{\epsilon,\delta}
$$
where~$D_{+}=D_{y}$ and $D_-=D_x$, and hence
\begin{equation*}
I_3= \sum_\pm \iint_{Q_T} |A(u)-A(v)| \, b_r^{\pm
(\mu_\pm-\nu_\pm)^\ast} \cdot D_\pm \phi^{\epsilon,\delta} \, \dif
w=:I_3^++I_3^-.
\end{equation*}

\medskip

\noindent {\bf 3.} {\em Cutting w.r.t.~$z$.} The computations of
this step are similar to the ones in the proof of
Theorem~\ref{th:nonlin}. For the reader's convenience, we estimate
$I_3^-+I_4^-$, the terms that was left to the reader in the
preceding proof.

For any measure $\tilde\mu$ we let $\tilde\mu_1=\tilde\mu_{|_{0<|z|
\leq r_1}}$ and write
$\tilde\mu=\tilde\mu_1+\tilde\mu_{|_{|z|>r_1}}$ for $r_1>r$. Then
\begin{align*}
I_4^-\leq  &\ \underbrace{\iint_{Q_{T}}\sgn(v-u)\, \mathcal{L}^{-(\mu_--\nu_-)_{1},r}[A(v(\cdot,t))](x) \,\phi^{\epsilon,\delta}\, \dif w}_{=:I_5^-} \\
& +\iint_{Q_{T}}\sgn(v-u)\, \mathcal{L}^{-(\mu_--\nu_-),r_1}[A(v(\cdot,t))](x) \,\phi^{\epsilon,\delta}\, \dif w.
\end{align*}
Recall that $-(\mu_--\nu_-)_{1}$ is a positive L\'evy measure by
\eqref{nat:cut-levy}, so we can apply Lemma~\ref{nat:ipp} with
$-(\mu_--\nu_-)_{1}$ instead of $\mu$ and $k=u(y,s)$ to find that
$$
I_5^- \leq \iint_{Q_{T}^2} |A(v)-A(u)|
\,\Levy^{-(\mu_--\nu_-)_{1}^\ast,r}[\phi^{\epsilon,\delta}(\cdot,t,y,s)](x)\,
\dif w
$$
and
\begin{align*} 
&I_3^-+I_5^-\leq\\
&
\iint_{Q_{T}^2} |A(v)-A(u)| \left( b_r^{-(\mu_--\nu_-)^\ast}  \cdot D_x \phi^{\epsilon,\delta} +\Levy^{-(\mu_--\nu_-)_1^\ast,r}[\phi^{\epsilon,\delta}(\cdot,t,y,s)](x) \right) \dif w.
\end{align*}
Easy computations then leads to
\begin{align*}
&\Levy^{-(\mu_--\nu_-)_1^{\ast},r}[\phi^{\epsilon,\delta}(\cdot,t,y,s)](x)\\
&=\theta_\delta(t-s) \int_{r<|z|\leq r_1}
\bar{\theta}_\epsilon(x-y-z)-\bar{\theta}_\epsilon(x-y)\, \dif
(\nu_--\mu_-)(z),
\end{align*}
and we can rewrite the nonlocal operator as follows,
\begin{equation*}
\begin{split}
& b_r^{-(\mu_--\nu_-)^\ast} \, \cdot D_x \phi^{\epsilon,\delta}+\Levy^{-(\mu_--\nu_-)_1^{\ast},r}[\phi^{\epsilon,\delta}(\cdot,t,y,s)](x)\\
& = \theta_\delta(t-s) \int_{r<|z|\leq r_1} \bar{\theta}_\epsilon(x-y-z)-\bar{\theta}_\epsilon(x-y)+z \cdot D \bar{\theta}_{\epsilon} (x-y) \, \dif (\nu_--\mu_-)(z) \\
& \quad -\theta_\delta(t-s) \, D \bar{\theta}_{\epsilon} (x-y) \cdot
\underbrace{\left(-b_r^{-(\mu_--\nu_-)^\ast}+\int_{r<|z|\leq r_1} z
\, \dif (\nu_--\mu_-)(z)\right)}_{=\sgn(r_1-1) \int_{r_1 \wedge (1
\vee r) <|z|\leq r_1 \vee 1} z \, \dif (\nu_--\mu_-)(z)}.
\end{split}
\end{equation*}
Compare this expression with \eqref{nat:cnl-6} that appear when
$I_3^+$ and $I_4^+$ are considered.

We add the different estimates and find that for all~$r_1>r$,
\begin{align*}
& I_3^-+I_4^- \\
&\leq \iint_{Q_{T}^2} |A(u)-A(v)|  \, \theta_\delta (t-s)\\
&\qquad\quad\cdot \int_{r<|z|\leq r_1} \bar{\theta}_\epsilon(x-y-z)-\bar{\theta}_\epsilon(x-y)+z \cdot D \bar{\theta}_{\epsilon} (x-y) \, \dif (\nu_--\mu_-)(z) \, \dif w\\
& \quad-\iint_{Q_{T}^2} |A(u)-A(v)| \, \theta_\delta (t-s) \, D
\bar{\theta}_\epsilon(x-y)\\
&\qquad\quad\cdot \sgn(r_1-1) \int_{r_1 \wedge (1 \vee r) <|z|\leq r_1 \vee 1} z \, \dif (\nu_--\mu_-)(z) \, \dif w\\
&\quad +\underbrace{\iint_{Q_{T}^2}\sgn(v-u)\,
  \mathcal{L}^{-(\mu_--\nu_-),r_1}[A(v(\cdot,t))](x)
  \,\phi^{\epsilon,\delta}\, \dif w}_{\leq \iint_{Q_{T}^2} \sgn(v-u)
  \,
  \mathcal{L}^{-(\mu_--\nu_-),r_1}[A(u(\cdot,s))](y)
  \,\phi^{\epsilon,\delta}\, \dif w\ \text{by Lemma~\ref{nat:pro-kato-d}}}\\
&=J_1^-+J_2^-+J_3^-.
\end{align*}
Similar arguments also lead to
\begin{align*}
& I_3^++I_4^+ \\
&\leq  \iint_{Q_{T}^2} |A(u)-A(v)|\theta_\delta (t-s)\\
& \qquad\quad \cdot  \int_{r<|z|\leq r_1}
\bar{\theta}_\epsilon(x-y+z)-\bar{\theta}_\epsilon(x-y)-z \cdot D
\bar{\theta}_{\epsilon} (x-y) \, \dif (\mu_+-\nu_+)(z) \, \dif w\\
&\quad +\iint_{Q_{T}^2} |A(u)-A(v)|\,\theta_\delta (t-s) \, D \bar{\theta}_\epsilon(x-y)\\
& \qquad\quad   \cdot \sgn(r_1-1) \int_{r_1 \wedge (1 \vee r) <|z|\leq r_1 \vee 1} z \, \dif (\mu_+-\nu_+)(z) \, \dif w\\
&\quad +\iint_{Q_{T}^2}\sgn(u-v)\, \mathcal{L}^{(\mu_+-\nu_+),r_1}[A(u(\cdot,s))](y) \,\phi^{\epsilon,\delta}\, \dif w,\\
& =:J_1^++J_2^++J_3^+.
\end{align*}

\medskip

\noindent {\bf 4.} {\em $L^1 \cap BV$-regularity.} We
estimate~$J_i^\pm$ ($i=1,\dots,3$). Almost all the computations have
already been done in the preceding proof. Indeed,~$J_1^\pm$ are of the
same form as in~\eqref{nat:cnl-8}--\eqref{nat:cnl-8-bis}, with the new
nonlinearity~$A$ and the new measures~$\pm (\mu_\pm-\nu_\pm)$. Arguing
as for~\eqref{nat:esti-sing-1} thus gives
\begin{equation*}
\begin{split}
\sum_\pm J_1^\pm \leq  & \frac{1}{2\epsilon} \,
\int_{\R^d}|D\tilde\theta_d|\dif x \\
&\qquad\cdot \underbrace{|A(u)|_{L^1(0,T;BV)}}_{\leq |u_0|_{BV(\R^d)} \, T \,
\|A'\|_{L^\infty(\R)} \mbox{ by~\eqref{nat:nonincrease}}} \, \int_{r<|z| \leq r_1} |z|^2 \, \dif
\underbrace{\sum_\pm \pm (\mu_\pm-\nu_\pm)}_{=|\mu-\nu| \mbox{
    by~\eqref{nat:cut-levy}}}(z).
\end{split}
\end{equation*}
Moreover,~$\sum_\pm (\mu_\pm-\nu_\pm)=\mu-\nu$ and hence
\begin{equation*}
\begin{split}
\sum_\pm J_2^\pm = &\ \iint_{Q_{T}^2} |A(u)-A(v)|\, \theta_\delta (t-s) \, D \bar{\theta}_\epsilon(x-y) \\
&\qquad\cdot \sgn(r_1-1) \int_{r_1 \wedge (1 \vee r) <|z|\leq r_1
\vee 1} z \, \dif (\mu-\nu)(z) \, \dif w.
\end{split}
\end{equation*}
This term is of the same form than~$J_2^+$ in~\eqref{nat:cnl-8} (or~$J_2^-$ in~\eqref{nat:cnl-8-bis}) with the new nonlinearity~$A$ and the new flux~$\int_{r_1 \wedge (1 \vee r) <|z|\leq r_1
\vee 1} z \, \dif (\mu-\nu)(z)$. Arguing as for~\eqref{nat:esti-sing-2} and using~\eqref{nat:nonincrease} thus give  
\begin{equation*}
\sum_\pm J_2^\pm \leq |u_0|_{BV(\R^d)} \, T \, \|A'\|_{L^\infty(\R)}
\bigg|\int_{r_1 \wedge (1 \vee r) < |z| \leq r_1 \vee 1} z \, \dif
(\mu-\nu)(z)\bigg|.
\end{equation*}
Finally, 
\begin{align*}
\sum_\pm J_3^\pm \leq 
T \, \|A'\|_{L^\infty(\R)} \int_{|z| \geq r_1}
\|u_0(\cdot+z)-u_0 \|_{L^1(\R^d)} \, \dif |\mu-\nu| (z).
\end{align*}


\medskip

\noindent {\bf 5.} {\em Conclusion.} The rest of the proof is the
same as for Theorem~\ref{th:nonlin};~i.e. we use the estimates on
$J_i^\pm$ to estimate $I_3+I_4 \leq \sum_{i=1}^3 \sum_\pm J_i^\pm$
in \eqref{nat:cnl-1-levy} and pass to limit and/or optimizes w.r.t.
the parameters $r,\epsilon,\delta>0$. The proof is complete.
\end{proof}

\appendix

\section{Technical computations}\label{nat:app-tech}

\begin{proof}[Proof of \eqref{nat:esti-sing-1} and \eqref{nat:esti-sing-2}]
We start by proving \eqref{nat:esti-sing-1} in the
  $+$ case. Similar arguments give the proof also in the $-$ case. From
  Taylor's formula
\begin{align*}
\bar{\theta}_{\epsilon}(x-y+z)-\bar{\theta}_{\epsilon}(x-y)-z \cdot
D \bar{\theta}_{\epsilon}(x-y) \, = \int_0^1  (1-\tau) \,
D^2\bar{\theta}_{\epsilon}(x-y+\tau \, z) \, z \cdot z \, \dif \tau,
\end{align*}
and hence by Fubini's theorem,~$J_1^+$ in~\eqref{nat:cnl-8} can be written as
\begin{multline}\label{nat:s-1}
J_1^+=\iint_{(0,T)^2}
\int_{r<|z|\leq r_1} \int_0^1  \theta_\delta (t-s) \, (1-\tau)  \\
\cdot \underbrace{\int_{\R^{d}} \int_{\R^d} |C_+(v(x,t))-C_+(u(y,s))|
\, D^2\bar{\theta}_{\epsilon}(x-y+\tau \, z) \, z \cdot z \, \dif y
\, \dif x}_{=:J} \, \dif \tau \, \dif \mu(z) \, \dif t \, \dif s.
\end{multline}
For any $k\in\R$, it is classical that
$\eta_k(C_+(u(\cdot,s)))=|k-C_+(u(\cdot,s))|\in  BV$ with
$$
|D\eta_k(C_+(u(\cdot,s))| \leq  |DC_+(u(\cdot,s))|,
$$
since it is the composition of a Lipschitz-continuous function with
a~$BV$-function; see e.g.~\cite{Perthame,EvGa92,Rud74}.  Integration
by parts w.r.t.~$y$ (for fixed~$z,x,t,s$), then leads to
\begin{align*}
|J|& = \left|\int_{\R^d} \int_{\R^d} D \bar{\theta}_{\epsilon}(x-y+\tau \, z) \cdot z \, z \cdot \dif D\eta_{C_+(v(x,t))}(C_+(u(\cdot,s)))(y) \, \dif x\right|,\\
& \leq  |z|^2 \int_{\R^d} \int_{\R^d} |D
\bar{\theta}_{\epsilon}(x-y+\tau \, z)| \, \dif |DC_+(u(\cdot,s))|(y)
\, \dif x.
\end{align*}
Note that by the definition of~$\bar \theta_\epsilon$ (just
below~\eqref{nat:test-kuznetsov}), we have 
$$
\int_{\R^d} |D
\bar{\theta}_{\epsilon}(x)| \, \dif x=\frac{1}{\epsilon}\,\int_{\R^d}|D\tilde\theta_d|\dif x.
$$
Hence, we change the order of integration (using Fubini) to see that
$$
|J| \leq  |z|^2 \, |C_+(u(s))|_{BV(\R^d)} \, \int_{\R^d} |D
\bar{\theta}_{\epsilon}(x)| \, \dif x \leq   |z|^2 \,
|C_+(u(s))|_{BV(\R^d)}\frac{1}{\epsilon} \,
\int_{\R^d}|D\tilde\theta_d|\dif x ,
$$
and then from \eqref{nat:s-1} that
\begin{align*}
 |J_1^+|\leq \, & \frac{1}{\epsilon}\,\int_{\R^d}|D\tilde\theta_d|\dif x\\
&\cdot\iint_{(0,T)^2} \int_{r<|z|\leq r_1} \int_0^1  \theta_\delta
(t-s) \, (1-\tau)  \, |z|^2 \,|C_+(u(s))|_{BV(\R^d)} \, \dif \tau \,
\dif \mu(z) \, \dif t \, \dif s.
\end{align*}
Let us recall that the integrand above is $\dif \tau \, \dif \mu(z)
\, \dif t \, \dif s$-measurable since $s \rightarrow
|u(s)|_{BV(\R^d)}$  is lower semi-continuous. By Fubini we then
integrate first w.r.t.~$t$ and use that~$\int_0^T \theta_\delta(t-s)
\,  \dif t \leq 1$ to see that
$$
|J_1^+| \leq \frac{1}{\epsilon} \, \int_{\R^d}|D\tilde\theta_d|\dif x
\int_0^1 (1-\tau) \, \dif \tau \int_{r<|z|\leq r_1} \, |z|^2 \dif
\mu(z)  \int_{0}^T  |C_+(u(s))|_{BV(\R^d)} \, \dif s,
$$
and the proof of \eqref{nat:esti-sing-1} is complete.

\medskip

We prove \eqref{nat:esti-sing-2} by similar arguments. Define
\begin{equation}\label{nat:s-3}
q(v,u):=|v-u| \, \sgn(r_1-1) \int_{r_1 \wedge (1 \vee r) <|z|\leq
r_1 \vee 1} z \, \dif \mu(z),
\end{equation}
and note that it is Lipschitz-continuous. Again by
Fubini's theorem,~$J_2^+$ in~\eqref{nat:cnl-8} can be written as
\begin{equation}\label{nat:s-2}
J_2^+=\iint_{(0,T)^2}  \theta_\delta (t-s) \underbrace{\int_{\R^{d}}
\int_{\R^d} D \bar{\theta}_\epsilon(x-y) \cdot q(C_+(v(x,t),C_+(u(y,s)))
\, \dif y \, \dif x}_{=:J} \,  \dif t \, \dif s.
\end{equation}
For fixed~$(x,t,s)$,~$q(C_+(v(x,t),\cdot)$ is Lipschitz-continuous and~$C_+(u(\cdot,s))$ is~$BV$; hence, the composition~$q(C_+(v(x,t)),C_+(u(\cdot,s)))$ is in~$BV(\R^d,\R^d)$ with 
$$
|\Div_y q(C_+(v(x,t)),C_+(u(\cdot,s)))| \leq |D C_+(u(\cdot,s))| \, \|q_u\|_{L^\infty(\R,\R^d)},
$$
where~$\|q_u\|_{L^\infty(\R,\R^d)}$ denotes the Lipschitz
constant of~$q$ w.r.t. its second variable.
We thus may integrate by parts in $y$
to see that
$$
|J| \leq \|q_u\|_{L^\infty(\R,\R^d)} \int_{\R^{d}} \int_{\R^d}
\bar{\theta}_{\epsilon}(x-y) \, \dif |D C_+(u(\cdot,s))|(y) \, \dif x.
$$
Changing the order
of integration, we find that
$$J \leq
|C_+(u(s))|_{BV(\R^d)} \|q_u\|_{L^\infty(\R,\R^d)}, $$ and hence
by \eqref{nat:s-2} and integrating first w.r.t.~$t$, we get that
\begin{equation}\label{nat:s-4}
|J_2^+| \leq \|q_u\|_{L^\infty(\R,\R^d)}  \int_0^T
|C_+(u(s))|_{BV(\R^d)} \, \dif s.
\end{equation}
The proof of~\eqref{nat:esti-sing-2} is now complete since by
\eqref{nat:s-3},
$$\|q_u\|_{L^\infty(\R,\R^d)} =\left|\int_{r_1 \wedge (1 \vee r) <|z|\leq r_1 \vee 1} z \, \dif \mu(z) \right|.$$
\end{proof}

\section*{Acknowledgement}

We would like to thank the two anonymous referees for their careful reports. They helped us improve the paper a lot.

\end{document}